\documentclass[oneside]{amsart}

\usepackage[letterpaper, total={6.5in, 8.5in}]{geometry}
\usepackage{amsmath}
\usepackage{amsfonts}
\usepackage{amssymb}
\usepackage[dvipsnames]{xcolor}
\usepackage{mathrsfs}
\usepackage{enumerate}
\usepackage{amsthm}
\usepackage{tikz}
\usepackage{pgfplots}
\usepackage{float}
\usepackage[pdfdisplaydoctitle=true,
            colorlinks=true,
            urlcolor=blue,
            citecolor=blue,
            linkcolor=blue,
            pdfstartview=FitH,
            pdfpagemode= UseNone,
            bookmarksnumbered=true]{hyperref}
\usetikzlibrary{decorations.pathreplacing}
\usepgfplotslibrary{fillbetween}
\pgfplotsset{compat=1.18}
\definecolor{darkgreen}{rgb}{0,0.5,0}
\definecolor{darkred}{rgb}{0.7, 0, 0}

\makeatletter
\def\@tocline#1#2#3#4#5#6#7{\relax
  \ifnum #1>\c@tocdepth 
  \else
    \par \addpenalty\@secpenalty\addvspace{#2}%
    \begingroup \hyphenpenalty\@M
    \@ifempty{#4}{%
      \@tempdima\csname r@tocindent\number#1\endcsname\relax
    }{%
      \@tempdima#4\relax
    }%
    \parindent\z@ \leftskip#3\relax \advance\leftskip\@tempdima\relax
    \rightskip\@pnumwidth plus4em \parfillskip-\@pnumwidth
    #5\leavevmode\hskip-\@tempdima
      \ifcase #1
       \or\or \hskip 1em \or \hskip 2em \else \hskip 3em \fi%
      #6\nobreak\relax
    \hfill\hbox to\@pnumwidth{\@tocpagenum{#7}}\par
    \nobreak
    \endgroup
  \fi}
\makeatother

\newtheorem{theorem}{Theorem}[section]
\newtheorem{lemma}[theorem]{Lemma}
\newtheorem{proposition}[theorem]{Proposition}
\newtheorem{corollary}[theorem]{Corollary}

\theoremstyle{definition}

\newtheorem{example}[theorem]{Example}

\newtheorem*{theorem*}{Theorem}
\theoremstyle{remark}
\newtheorem{remark}[theorem]{Remark}

\numberwithin{equation}{section}
\newcommand{\R}{\mathbb{R}}
\newcommand{\T}{\mathbb{T}}
\newcommand{\N}{\mathbb{N}}
\newcommand{\Z}{\mathbb{Z}}

\newcommand{\Dmu}{\mathscr{D}_{\mu}}

\newcommand{\id}{\mathrm{id}}

\newcommand{\symp}{\nabla^\perp}

\renewcommand{\div}{\mathrm{div}}
\DeclareMathOperator{\Ad}{Ad}
\DeclareMathOperator{\ad}{ad}
\DeclareMathOperator{\adstar}{\ad^\star\!\!}
\DeclareMathOperator{\Adstar}{\Ad^{\star}\!\!}

\newcommand{\dist}{\mathrm{dist}}
\newcommand{\U}{\mathcal{U}}

\newcommand{\abs}[1]{\lvert#1\rvert}
\newcommand{\norm}[1]{\lVert#1\rVert}
\newcommand{\ip}[1]{\langle #1 \rangle}

\newcommand{\Bignorm}[1]{\Big\lVert#1\Big\rVert}

\newcommand{\bigip}[1]{\big\langle #1 \big\rangle}

\newcommand{\Id}{\mathrm{Id}}                         
\newcommand{\admiss}{\sigma}                          
\newcommand{\reducedadmiss}{\alpha}                
\newcommand{\translationpath}{\xi_{\ell}}         
\newcommand{\geod}{\gamma}                            
\newcommand{\riemexp}{\exp_e}                         
\newcommand{\normalball}[2][s]{B^{#1}(0,#2)}          
\newcommand{\normalhood}[3][s]{\U^{#1}(#2,#3)}       
\newcommand{\leray}{P}                               
\newcommand{\gradproj}{Q}                            
\newcommand{\length}[1]{\mathcal{L}(#1)}   
\newcommand{\action}[1]{\mathcal{A}(#1)}   
\newcommand{\geodistance}{\dist_{L^2}^{\Dmu^s}}      
\newcommand{\sqdistance}{\Phi}                        
\newcommand{\lift}[1]{\tilde{#1}}                 
\newcommand{\Ltwodiv}{L^2_{\mathrm{div}}}             
\newcommand{\opnorm}[2]{\norm{#1}_{\mathrm{op}(#2)}}  
\newcommand{\closedball}[2][s]{\overline{B^{#1}}(0,#2)}          
\newcommand{\closedhood}[3][s]{\overline{\U^{#1}}(#2,#3)}        

\author{Theodore D. Drivas}
\author{Patrick Heslin}
\author{Gerard Misio{\l}ek}

\address{T.D. Drivas: Department of Mathematics, Stony Brook University, Stony Brook, NY, 11790}
\email{tdrivas@math.stonybrook.edu}

\address{P. Heslin: Max Planck Institute for Mathematics in the Sciences, 04103, Leipzig, Germany}
\email{patrick.heslin@mis.mpg.de}

\address{G. Misio{\l}ek: University of Notre Dame, Notre Dame, Indiana, 46556}
\email{gmisiole@nd.edu}

\begin{document}

\title{Distance Geometry and a Hydrodynamical Hopf-Rinow Criterion}

\subjclass[2020]{Primary 58D05; Secondary 35Q31.}
\keywords{Euler equations, volume-preserving diffeomorphisms, minimizing geodesics, Hopf-Rinow theorem, metric betweenness}

\begin{abstract}
    The hydrodynamical Hopf-Rinow problem asks whether a given reconfiguration of an ideal fluid can be attained by an Euler flow. On the flat two-dimensional torus, we give an equivalent criterion for the strong version of this problem (which requires that the fluid motion minimizes the total kinetic energy required) in terms of metric betweenness. Toward this, we prove that sufficiently small normal neighborhoods in the volume-preserving diffeomorphism group of the flat $n$-dimensional torus, $n \geq 2$, equipped with the kinetic energy metric, are geodesically convex; resolving the strong problem locally in the $H^s$ topology. We further establish local bi-Lipschitz equivalence of the intrinsic geodesic distance and the extrinsic $L^2$ distance. Finally, we prove that the squared geodesic distance from the identity admits a $C^{1,1}$ extension from a sufficiently small normal neighborhood to the ambient $L^2$ space.
\end{abstract}

\maketitle

\tableofcontents

\section{Introduction}
Arnold's \cite{arnold1966sur} geometric formulation of the incompressible Euler equations interprets the motion of an ideal fluid as a geodesic flow on the group of volume-preserving diffeomorphisms of the domain equipped with the natural kinetic energy metric. Ebin and Marsden \cite{ebin1970groups} developed the analytic framework for the above interpretation in the Sobolev $H^s$ setting. They showed that the geodesic equation is a smooth ordinary differential equation in the $H^s$ topology and, consequently, the $L^2$ metric admits a smooth Riemannian exponential map which is a local diffeomorphism. Hence, given a sufficiently regular initial velocity, the associated Cauchy problem has a unique solution, at least for short time.

\enlargethispage{2\baselineskip}

The two-point boundary-value problem asks the following question. Given two attainable fluid configurations, does there exist an Euler flow connecting them? Here attainable means that the configurations can be joined by an admissible (piecewise regular) curve of finite $L^2$ length. We distinguish two versions of this problem. The first version asks only for an Euler flow with the prescribed endpoints, or equivalently whether the $L^2$-Riemannian exponential map covers the connected component of the identity. The second version asks for an $L^2$ length-minimizing geodesic. We refer to these as the weak and strong hydrodynamical Hopf-Rinow problems, respectively. Recall that in finite dimensions the Hopf-Rinow theorem gives both conclusions under the equivalent assumptions of metric and geodesic completeness. These conclusions need not persist in infinite dimensions, where the local compactness used in the classical argument is no longer available \cite{atkin1975the,grossman1965hilbert, mcalpin1965infinite}. However, the local constructions in \cite{ebin1970groups} establish an affirmative answer to the weak version of the problem for fluid configurations which are sufficiently close to the identity in the $H^s$ topology. 

An attainable fluid configuration of the three-dimensional cube which cannot be reached by an $L^2$ length-minimizing geodesic was constructed by Shnirelman in \cite{shnirelman1985the}, presenting a negative answer to the strong problem in this setting. Brenier \cite{brenier1989the} subsequently recovered existence of minimizing paths after enlarging the admissible class to generalized incompressible flows and Shnirelman \cite{shnirelman1994generalized} developed an approximation technique using classical flows in dimensions at least three. These results give a relaxed variational theory but do not in general produce minimizing geodesics in the Sobolev framework. On the flat two-dimensional torus both the weak and strong global problems remain open.

There has nevertheless been substantial progress on the structure and properties of the $L^2$ exponential map. For fluid motion on a closed surface, it was shown to be nonlinear Fredholm of index zero \cite{ebin2006singularities}, with further developments made in \cite{li2022on, shnirelman2005microglobal}. Subsequent work established that for two-dimensional fluids it is a covering space map onto an open connected component of the identity with infinite diameter \cite{brandenbursky2017the, misiolek2010fredholm}. The regularity of the two-point boundary-value problem was studied in \cite{heslin2023two}, where Euler flows on the flat two-torus were shown to inherit the regularity of their endpoints. These results yield information about the global properties of the $L^2$ exponential map without proving its surjectivity. Accounts of the surrounding literature can be found in the monograph of Arnold and Khesin \cite{arnold2021topological} as well as the recent surveys \cite{drivas2023singularity,elgindi2025dynamics,khesin2023open}.

An immediate question concerning local geometric properties arises from these global considerations. The inverse function theorem gives $H^s$ normal coordinates and the Gauss lemma shows that a radial\footnote{Here radial refers to emanating from the center of the normal neighborhood.} geodesic minimizes among competitors which remain in the normal neighborhood. However, an $H^s$ normal neighborhood does not contain any ball for the distance induced by the $L^2$ metric. Consequently, the classical Riemannian argument does not compare the radial geodesic with competitors which leave the normal neighborhood. This issue was already left open in the discussion of the local variational problem in Ebin and Marsden; cf.\ \cite[Section~15.2(viii)]{ebin1970groups}.

Our first main result addresses this issue for the flat torus. In Theorem~\ref{local geodesic convexity} we prove that for sufficiently small normal neighborhoods the radial geodesics are the unique $L^2$ length minimizers among all admissible curves connecting the identity to their endpoint. We emphasize that the competitor curves are not required to remain in the normal neighborhood. The proof adapts Brenier's \cite{brenier1989the} sufficient condition for an Euler flow to minimize the action, see also the exposition in \cite{drivas2023singularity}. Since the fluid domain is the torus, an additional argument is needed to remove the potential integer winding of lifted competitors. We further strengthen this result in Theorem~\ref{geodesic convexity of normal neighborhoods}, where we show that sufficiently small normal neighborhoods in the volume-preserving diffeomorphism group are in fact geodesically convex: every pair of points is joined by an Euler flow contained entirely in the same neighborhood and this geodesic is, up to reparameterization, the unique length-minimizer among all admissible curves with those endpoints. This establishes a fundamental local Riemannian-geometric tool for the kinetic energy metric.

We next investigate the relationship between the intrinsic geodesic distance in the diffeomorphism group and the extrinsic distance in the ambient $L^2$ space between particle configurations. The constructions in \cite{shnirelman1994generalized} give a global comparison result for fluid configurations of cubes in dimensions at least three, with recent improvements in \cite{schiffer2024on}. By a different argument we prove in Theorem \ref{local Shnirelman inequality} an analogous local comparison result for the flat $n$-dimensional torus $\T^n$ which holds in every dimension $n\geq2$. This is extended in Corollary~\ref{Phi Lipschitz} to local bi-Lipschitz equivalence of the two distance metrics. This comparison also reveals a compactness property which is important in the global argument: although closed $H^s$ balls are not strongly compact, Proposition~\ref{compactness of closed normal neighborhood} shows that their image under the $L^2$-exponential map, for sufficiently small radii, is compact with respect to the geodesic distance. We then address the natural question of whether the local equivalence of metrics is a consequence of a bound on the second fundamental form, as it would be for a finite-dimensional submanifold of Euclidean space. In the hydrodynamical setting the pressure gradient is the formal analogue of the normal acceleration. Proposition~\ref{local ratio unbounded} shows, however, that the corresponding quadratic map is not locally bounded in $L^2$. Hence the comparison of these distances is not explained by a naive finite-dimensional curvature argument.

\clearpage

We then return to the global two-point boundary value problem. Our principal result is Theorem~\ref{Hopf-Rinow criteria} which gives a criterion in the two-dimensional setting for the strong hydrodynamical Hopf-Rinow problem in terms of Menger's \cite{menger1928untersuchungen} notion of metric betweenness. In particular we prove that, for a sufficiently small closed normal neighborhood, it is equivalent to the statement that every attainable configuration outside of the closed normal neighborhood has a point, distinct from the identity, inside that neighborhood which is metrically between the attainable configuration and the identity. Targets inside the closed normal neighborhood are already covered by Theorem~\ref{local geodesic convexity}. This can be viewed as a global-to-(quasi)-local reduction: the betweenness condition still depends on the global distance to the target but the required intermediate points lie in one fixed neighborhood, compact with respect to the geodesic distance.

This formulation also highlights the obstruction to extending the classical finite-dimensional proof. There one takes first-exit points of almost-minimizing curves on a sufficiently small metric sphere. Compactness produces a limiting boundary point and continuity of the geodesic distance ensures that this point is metrically between the center and the target. In our weak setting, the boundary of an $H^s$ normal neighborhood is not a sphere for the $L^2$ geodesic distance. The weak metric behaves as a distorted lens: a boundary point may be closer to the target than other boundary points without being metrically between the identity and the target. Hence minimizing distance to the target and satisfying metric betweenness need not agree. Once the correct metric notion has been identified, one must still overcome the lack of compactness of the boundary sphere. We conclude the section with an example of an infinite dimensional ellipsoid which demonstrates even harsher pathological behavior.

Finally, we prove in Theorem~\ref{Whitney extension of squared distance} that the squared geodesic distance admits a $C^{1, 1}$ extension to the ambient $L^2$ space via a version of the Whitney extension theorem. However, we show in Proposition \ref{no C4 extension with prescribed differential} that there exists no $C^4$ extension which preserves the values and the ambient differential. The section contains several quantitative estimates for the $L^2$ exponential map in the weak topology. As a further consequence, we prove that the intrinsic and ambient distances studied in Section \ref{local metric geometry} agree to first order at every point of a sufficiently small normal neighborhood. Lastly, we derive a third variational formula for the $L^2$ energy functional which provides an alternative avenue of proof for the mixed Sobolev estimates.

The paper is organized as follows. In Section~\ref{background} we recall the necessary analytic and geometric background. In Section~\ref{minimizing geodesics} we prove geodesic convexity of sufficiently small normal neighborhoods. Section~\ref{local metric geometry} contains the proof of the local equivalence of the geodesic and ambient $L^2$ distances and the compactness of closed normal neighborhoods as well as an explanation of the obstruction to a naive second-fundamental-form argument. In Section~\ref{Hopf-Rinow} we formulate the aforementioned metric betweenness criterion for the strong hydrodynamical Hopf-Rinow problem on $\T^2$ and discuss the corresponding infinite-dimensional obstructions. Section~\ref{regularity of squared geodesic distance} contains the proof of the ambient $C^{1,1}$ Whitney extension theorem for the squared geodesic distance and the discussion of the second derivative of the $L^2$ exponential map.
\vfill
\subsection*{Acknowledgements} GM thanks the Institute for Mathematical Sciences at Stony Brook University for their hospitality and support while part of this work was carried out. PH similarly thanks the Department of Mathematics at Stony Brook University for their hospitality.  TDD was partially supported by an NSF CAREER award \#2235395, a Stony Brook University Trustee’s award, an Alfred P. Sloan Fellowship. The authors also express their gratitude to Giacomo Del Nin for helpful conversations and suggestions regarding Section \ref{regularity of squared geodesic distance}.

\newpage

\section{Technical background}\label{background}

Throughout the paper we work on the flat $n$-torus $M = \T^n \simeq \R^n / \Z^n$ with $n \geq 2$. We write $dx$ for the volume form $\mu$ on $\T^n$ and $\nabla$ for the corresponding Riemannian connection. In what follows we will assume that $\frac{n}{2} + 1 < s' < s$. The letter $C$ will denote a positive constant whose dependencies will be indicated wherever relevant.

Let $H^s(\T^n,\T^n)$ denote the set of maps of Sobolev class $H^s$. The group of volume-preserving diffeomorphisms of Sobolev class
\begin{equation*}
\Dmu^s(\T^n) =\bigl\{\eta\in H^s(\T^n,\T^n): \eta^{-1}\in H^s(\T^n,\T^n),\ \eta^*dx=dx\bigr\}
\end{equation*}
is a smooth Hilbert manifold and a topological group under composition. Its tangent space at the identity diffeomorphism is
\begin{equation*}
T_e\Dmu^s(\T^n) =\bigl\{u\in H^s(\T^n,\R^n):\div \, u=0\bigr\}.
\end{equation*}
For an arbitrary volume-preserving diffeomorphism $\eta$ we define
\begin{equation*}
T_\eta\Dmu^s(\T^n) = \{u\circ\eta:u\in T_e\Dmu^s(\T^n)\}.
\end{equation*}
When inverse Laplacians are used, we impose the usual zero-mean condition. We use the sign convention for which $-\Delta$ is nonnegative.

Following Arnold's geometric approach to the study of ideal fluid motion we equip $T_e\Dmu^s(\T^n)$ with an $L^2$ inner product
\begin{equation}\label{L^2 metric}
\ip{u,v}_{L^2} = \int_{\T^n}u(x)\cdot v(x)\,dx, \qquad u,v\in T_e\Dmu^s(\T^n)
\end{equation}
which we extend to a weak\footnote{This metric is weak in the sense that the topology it induces is coarser than the underlying $H^s$ Hilbert manifold topology.} Riemannian metric on $\Dmu^s(\T^n)$ by right-translation.

The geodesic equations of this metric correspond to the Euler equations of ideal hydrodynamics. In the Lagrangian representation the Cauchy problem reads
\begin{equation}\label{euler in lagrangian}
\begin{split}
&\ddot\geod(t)=-\nabla p(t)\circ\geod(t),\\
&\dot\geod(0)=u_0,\\
&\geod(0)=e.
\end{split}
\end{equation}
This second order system, under the flow equation
\begin{equation*}
    \frac{d\gamma}{dt}(t,x) = u(t, \gamma(t,x)), \quad \gamma(0, x) = x
\end{equation*}
reduces to the familiar Eulerian representation
\begin{equation}\label{euler}
\begin{split}
&\partial_t u+\nabla_u u=-\nabla p,\\
&\div \, u=0,\\
&u(0)=u_0,
\end{split}
\end{equation}
where $p$ is the pressure function with zero mean
\begin{equation*}
\int_{\T^n}p(t,x)\,dx=0.
\end{equation*}

Ebin and Marsden \cite{ebin1970groups} showed that \eqref{euler in lagrangian} is a smooth second-order ordinary differential equation on $\Dmu^s(\T^n)$. Consequently, the $L^2$ metric \eqref{L^2 metric} admits a smooth Riemannian exponential map $\exp_e$ defined on an open $H^s$ neighborhood of the origin. Moreover we have
\begin{equation*}
\exp_e(0)=e \quad \text{and} \quad d\exp_e(0)=\mathrm{Id}
\end{equation*}
and therefore, by the inverse function theorem, there exists $\varepsilon>0$ such that
\begin{equation}\label{normal neighborhood}
\exp_e:B^s(0,\varepsilon) \rightarrow \U^s(e,\varepsilon)\subset\Dmu^s(\T^n)
\end{equation}
is a diffeomorphism where 
\begin{equation}\label{normal ball}
    B^s(0,\varepsilon) = \big\{v \in T_e\Dmu^s(\T^n) \ : \ \norm{v}_{H^s} < \varepsilon\big\}
\end{equation}
and
\begin{equation}\label{normal hood}
    \normalhood{e}{\varepsilon} = \riemexp \normalball{\varepsilon}
\end{equation}
is an $H^s$-open set containing the identity. We will refer to such $\normalhood{e}{\varepsilon}$ as a normal neighborhood. Whenever $0<\varepsilon<r$ and $\riemexp$ is a diffeomorphism on $\normalball{r}$, we further define
\begin{equation*}
    \closedball{\varepsilon} = \bigl\{u \in T_e\Dmu^s(\T^n):\norm{u}_{H^s}\leq\varepsilon\bigr\}, \qquad \closedhood{e}{\varepsilon}=\riemexp \closedball{\varepsilon}.
\end{equation*}
We refer to $\closedhood{e}{\varepsilon}$ as a closed normal neighborhood. Throughout, the notation $\exp_e^{-1}$ represents the inverse of the exponential map restricted to the ball $B^s(0, \varepsilon)$. Namely,
\begin{equation*}
\exp_e^{-1} = \big(\exp_e\big|_{B^s(0,\varepsilon)}\big)^{-1}.
\end{equation*}
When closed normal neighborhoods are under consideration, $\riemexp^{-1}$ will denote the inverse restricted to an open normal neighborhood of some larger radius.

The Riemannian exponential map corresponds to the Lagrangian data-to-solution map of the fluid equations. More concretely, for an initial velocity $u_0\in T_e\Dmu^s(\T^n)$ the $L^2$-geodesic 
\begin{equation*}
\geod(t) = \exp_e tu_0
\end{equation*}
satisfies $\dot\geod(t)=u(t)\circ\geod(t)$ and $\geod(0) = e$ with $u(t)$ solving the Cauchy problem \eqref{euler}.

Alongside the Riemannian structure we will also use a Lie-theoretic framework. Recall that, given $\eta\in\Dmu^s(\T^n)$, the Lie group adjoint is
\begin{equation*}
\Ad_\eta v = \eta_*v = ( D\eta \, v )\circ\eta^{-1}
\end{equation*}
and the associated Lie algebra adjoint is
\begin{equation}\label{ad convention}
\ad_u v = -[u,v] = \nabla_v u - \nabla_u v
\end{equation}
The $L^2$ coadjoint operators are given by
\begin{equation*}
\ip{\Adstar_\eta v,w}_{L^2} = \ip{v,\Ad_\eta w}_{L^2} \quad \text{and} \quad \ip{\adstar_u v,w}_{L^2} = \ip{v,\ad_u w}_{L^2}.
\end{equation*}

With this notation the Euler equation in \eqref{euler} can be written in the Euler-Arnold form as
\begin{equation}\label{Euler Arnold}
\begin{split}
\partial_t u &= -\adstar_u u, \\
u(0) &= u_0.
\end{split}
\end{equation}
Along any geodesic $\geod(t)$ there holds a coadjoint conservation law of the form
\begin{equation}\label{coadjoint conservation}
\Adstar_{\geod(t)}u(t)=u_0,
\end{equation}
or equivalently,
\begin{equation*}
u(t)=\Adstar_{\geod(t)^{-1}}u_0.
\end{equation*}
On the flat torus, this last identity may be recognized as the Weber formula
\begin{equation}\label{Weber formula prelim}
u(t) = P_e \Bigl( \big(D\geod(t)^{-T}u_0\big)\circ\geod(t)^{-1} \Bigr),
\end{equation}
where
\begin{equation*}
P_e = \mathrm{Id}- Q_e \quad \text{and} \quad Q_e = \nabla\Delta^{-1}\div
\end{equation*}
are the Leray projections onto the divergence-free and gradient fields respectively.

In what follows we will need the following estimate for the Lie algebra coadjoint.

\begin{lemma}\label{algebra coadjoint estimate}
There exists a constant $C>0$ such that, for all $v, w \in T_e\Dmu^s(\T^n)$ we have
\begin{equation*}
\norm{\adstar_w v}_{L^2} \leq C\norm{v}_{H^{s}}\norm{w}_{L^2}.
\end{equation*}
\end{lemma}

\begin{proof}
For divergence-free fields on the flat torus we have
\begin{equation*}
\adstar_w v = P_e \bigl( \nabla_w v + D w^{T} v \bigr).
\end{equation*}
Since $\nabla (w \cdot v) = D w^T v + D v^T w$ and $P_e$ annihilates gradients we have
\begin{equation*}
\adstar_w v = P_e \bigl( \nabla_w v - D v^T w \bigr)
\end{equation*}
and hence
\begin{equation*}
\norm{\adstar_wv}_{L^2} \leq C \norm{\nabla v}_{L^\infty}\norm{w}_{L^2} \leq C \norm{v}_{H^{s}}\norm{w}_{L^2}.
\end{equation*}
\end{proof}

We will refer to a curve as admissible if it is piecewise $C^1$ and regular; however, we will also allow for constant curves when necessary. For such a curve $\sigma:[0,1]\to\Dmu^s(\T^n)$ we define its $L^2$ length
\begin{equation}\label{length}
\length\sigma = \int_0^1\norm{\dot\sigma(t)\circ\sigma(t)^{-1}}_{L^2}\,dt
\end{equation}
and its action
\begin{equation}\label{action}
\action\sigma = \frac{1}{2} \int_0^1\norm{\dot\sigma(t)\circ\sigma(t)^{-1}}_{L^2}^2\,dt.
\end{equation}
The induced $L^2$-geodesic distance between two volume-preserving diffeomorphisms in $\Dmu^s(\T^n)$ is then
\begin{equation}\label{distance}
\geodistance(\xi,\eta) = \inf_{\substack{\sigma\text{ admissible} \\ \sigma(0)=\xi,\ \sigma(1)=\eta }} \mathcal L(\sigma)
\end{equation}
and, since the $L^2$ metric is right-invariant we have
\begin{equation*}
\geodistance(\xi,\eta) = \geodistance(e,\eta\circ\xi^{-1}).
\end{equation*}

We say $\eta$ is attainable from the identity when $\geodistance(e,\eta)<\infty$. Lastly, we record that this distance is indeed non-degenerate.

\begin{lemma}\label{nondegeneracy of geodesic distance}
For all $\eta,\varphi\in\Dmu^s(\T^n)$ which can be joined by an admissible curve, one has
\begin{equation}\label{global map distance lower bound}
\biggl(\int_{\T^n}\dist^{\T^n}\big(\eta(x),\varphi(x)\big)^2\,dx\biggr)^{1/2}\leq\geodistance(\eta,\varphi),
\end{equation}
where $\dist^{\T^n}$ is the flat distance on the torus. In particular, $\geodistance$ is nondegenerate on every path component of $\Dmu^s(\T^n)$.
\end{lemma}
\begin{proof}
Let $\admiss:[0,1]\to\Dmu^s(\T^n)$ be an admissible curve from $\eta$ to $\varphi$. For every $x\in\T^n$ we have
\begin{equation*}
\dist^{\T^n}\big(\eta(x),\varphi(x)\big)\leq\int_0^1\abs{\partial_t\admiss(t,x)}\,dt.
\end{equation*}
Minkowski's inequality and volume preservation give
\begin{equation*}
\biggl(\int_{\T^n}\dist^{\T^n}\big(\eta(x),\varphi(x)\big)^2\,dx\biggr)^{1/2}\leq\int_0^1\norm{\partial_t\admiss(t)}_{L^2}\,dt=\length{\admiss}.
\end{equation*}
Taking the infimum over all admissible curves proves \eqref{global map distance lower bound}. If the left hand side vanishes, then $\eta=\varphi$ almost everywhere and hence everywhere by continuity.
\end{proof}

\section{Geodesic convexity of normal neighborhoods}\label{minimizing geodesics}
In this section we prove that, for sufficiently small $H^s$ neighborhoods of the identity, the radial geodesics are the unique (up to reparameterization) length-minimizers among all admissible curves. More precisely we establish the following.
\begin{theorem}[Minimizing radial geodesics]\label{local geodesic convexity}
Let $n\geq2$ and $s>\frac n2+1$. There exists $\varepsilon>0$ such that for all $\eta \in \normalhood{e}{\varepsilon}$ the radial geodesic $t \mapsto \geod(t) = \exp_e(tu_0)$ with $\eta = \exp_e u_0$ and $u_0 \in \normalball{\varepsilon}$ is, up to reparameterization, the unique $L^2$ length-minimizer among all admissible curves in $\Dmu^s(\T^n)$ from $e$ to $\eta$. Consequently, for all $\eta \in \normalhood{e}{\varepsilon}$ we have
\begin{equation*}
\geodistance(e, \eta) = \norm{\exp_e^{-1} \eta}_{L^2} = \norm{u_0}_{L^2}.
\end{equation*}
\end{theorem}

\begin{remark}
Note that, by right-invariance, this establishes the same result for radial
geodesics in sufficiently small neighborhoods based at any point in the group.
\end{remark}

\begin{remark}
For a weak Riemannian metric one must distinguish length minimality among curves within a normal neighborhood from length minimality among all admissible curves. The Gauss lemma remains valid for smooth variations in the strong $H^s$ topology and shows that a radial geodesic minimizes among curves which remain in the $H^s$-normal neighborhood. It does not, by itself, compare the radial geodesic with curves which leave that neighborhood. Indeed, an $H^s$ normal neighborhood need not contain any ball for the distance induced by the weak $L^2$ metric. This distinction has also been noted in the original paper of Ebin and Marsden, cf. \cite[Section~15.2 (viii)]{ebin1970groups}. More generally, weak Riemannian metrics may even have degenerate geodesic distance, cf. \cite{eliashberg1993bi,michor2005vanishing}.

The obstruction is illustrated below. A curve $\alpha:[0,1]\rightarrow\Dmu^s(\T^n)$ may exit the neighborhood $\normalhood{e}{\varepsilon}$ in \eqref{normal neighborhood} at a diffeomorphism $\varphi_1$ which is an $H^s$-distance comparable to $\varepsilon$ from the identity but is much closer to it in $L^2$. Similarly, the curve may reenter $\normalhood{e}{\varepsilon}$ at a diffeomorphism $\varphi_2$ which is very close to the endpoint $\eta$ in $L^2$. Hence the portions of $\alpha$ lying inside the normal neighborhood have no a priori $L^2$ length comparable to the size of that neighborhood.

Theorem \ref{local geodesic convexity} is an unrestricted minimality statement: the competitors in the theorem are not required to remain in the normal neighborhood.
\end{remark}

\begin{figure}[H]
\centering
\begin{tikzpicture}[scale=1.5]

\definecolor{lightgray}{gray}{0.9}

\draw[thick] (-4.0, 0.25) to[out=30, in=150] (3.0, 0.25) 
    node[below=1pt] {$\Dmu^s$};

\filldraw[fill=lightgray, draw=black, very thick, densely dashed] (-0.5, -0.4) ellipse (2.5cm and 0.8cm);

\node[black, above=2pt] at (-0.5, -0.2) {$\mathcal{U}^s(e, \varepsilon)$};

\draw[black, ultra thick] 
    (-1.5, -0.4) 
    .. controls (-1.8, -0.6) and (-1.8, -0.8) .. (-1.7, -1.102) 
    .. controls (-1.5, -1.5) and (-1.2, -1.5) .. (-0.6, -1.45)  
    .. controls (0.2, -1.4) and (0.3, -1.6) .. (1.5, -1.4)      
    .. controls (1.5, -1.4) and (2.5, -1.3) .. (1.8177, -0.71)    
    .. controls (1.8177, -0.71) and (1.544, -0.551) .. (1.3, -0.4);     

\node[black, below=4pt] at (0.5, -1.45) {$\alpha$};

\fill[black] (-1.5, -0.4) circle (2pt) node[below right=1pt] {$e$};
\fill[black] (1.3, -0.4) circle (2pt) node[above right=1pt] {$\eta$};

\fill[black] (-1.7, -1.102) circle (2pt) node[below left=0pt] {$\varphi_1$};
\fill[black] (1.816, -0.71034) circle (2pt) node[right=1pt] {\ $\varphi_2$};

\end{tikzpicture}
\caption{A competitor curve $\alpha$ may exit and reenter the normal neighborhood at some $\varphi_1$ and $\varphi_2$ which are close to $e$ and $\eta$ respectively in the $L^2$ geodesic distance.}\label{gauss lemma issue}
\end{figure}
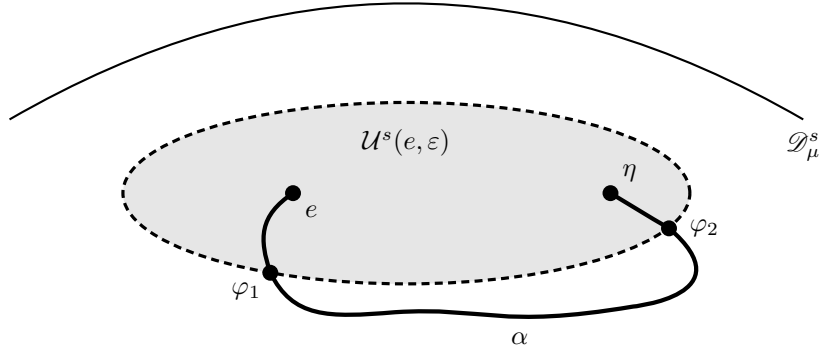

\enlargethispage{3\baselineskip}

For the proof of Theorem \ref{local geodesic convexity} we adapt the following criterion of Brenier, cf. \cite{brenier1989the, drivas2023singularity}.

\begin{lemma}
Let $M \subset \R^n$ be a bounded convex domain, possibly with boundary. For $u_0 \in T_e\Dmu^s(M)$, let $\geod(t) = \exp_e(tu_0)$ and $u(t) = \dot{\geod}(t) \circ \geod^{-1}(t)$ be the corresponding Lagrangian and Eulerian solutions of the fluid equations. If
\begin{equation}\label{pressure hessian condition}
\sup_{t \in [0,1]} \sup_{x \in M} \sup_{\abs{y}=1} \nabla^2 p(t,x)(y,y) < \pi^2
\end{equation}
where $p$ is the pressure in \eqref{euler}, then $\geod(t)$ is, up to reparametrization, the unique $L^2$ length minimizer among all admissible curves connecting $e$ to $\geod(1)$ in $\Dmu^s(M)$.
\end{lemma}

Note that the criterion above, as stated, cannot be naively applied to more complicated fluid domains. For example, any translation on the torus is pressureless, but eventually ceases to be globally length-minimizing. To modify the criterion for the torus, we reduce the comparison with arbitrary curves to a comparison with curves whose lifts to the universal cover $\R^n$ have the same endpoints.

\begin{lemma}\label{remove integer bulk}
Let $\eta \in \U^s(e, \varepsilon)$ with $u_0 \in B^s(0, \varepsilon)$ such that $\eta = \exp_e u_0$. Assume that the corresponding radial geodesic $\geod(t) = \exp_e tu_0$ admits a lift $\tilde{\geod}(t)$ to the universal cover with $\tilde{\geod}(t, x) = x + \zeta(t,x)$ with $\zeta(t, x) \in \Bigl(-\frac{1}{2}, \frac{1}{2}\Bigr)^n$ for all $(t,x) \in [0,1] \times \R^n$. Let $\sigma$ be an admissible curve from $e$ to $\eta$ in $\Dmu^s(\T^n)$ such that its lift $\tilde{\sigma}$ satisfies $\tilde{\sigma}(0,x) = x$ and $\tilde{\sigma}(1, x) = \tilde{\eta}(x) + \ell$ for some $\ell \in \Z^n$. Then there exists an admissible curve $\reducedadmiss : [0,1] \rightarrow \Dmu^s(\T^n)$ from $e$ to $\eta$ whose lift $\tilde{\reducedadmiss}$ beginning at $\tilde{\reducedadmiss}(0,x) = x$ satisfies $\tilde{\reducedadmiss}(1,x) = \tilde{\eta}(x)$. Furthermore we have $\action{\reducedadmiss} \leq \action{\sigma}$ with equality if and only if $\ell = 0$.
\end{lemma}

\begin{proof}
By the Hodge Decomposition Theorem the Eulerian velocity of $\sigma$ can be written as
\begin{equation*}
\dot{\sigma}(t, \sigma^{-1}(t,x)) = v(t,x) + h(t)
\end{equation*}
where $v(t)$ has zero mean and $h(t)$ is harmonic and spatially constant. We compute the bulk translation
\begin{align*}
\int_0^1 h(t)\,dt
&= \int_0^1 \int_{\T^n} v(t,x) + h(t) \, dx \, dt \\
&= \int_0^1 \int_{\T^n} \partial_t \tilde{\sigma}(t,x)\,dx\,dt \\
&= \int_{\T^n} \big(\tilde{\sigma}(1,x)-\tilde{\sigma}(0,x)\big)\,dx \\
&= \ell + \int_{\T^n} \zeta(1,x)\,dx.
\end{align*}
Letting $\displaystyle c = \int_{\T^n} \zeta(1,x)\,dx \in \bigl(-\tfrac{1}{2}, \tfrac{1}{2}\bigr)^n$, the goal is now to remove the integer term $\ell$.

Let $\translationpath(t,x) = x + t\ell$ denote the flow of the constant field $\ell$ and consider the decomposition
\begin{align*}
    v(t,x) + h(t) &= \ell + \Ad_{\translationpath(t)}\Ad_{\translationpath(t)}^{-1}\big(v(t,x) + h(t) - \ell\big) \\
    & = \ell + \Ad_{\translationpath(t)}\big(v(t,x+ \ell t) + h(t) - \ell\big).
\end{align*}
Defining $\reducedadmiss$ as the flow of $v(t,x+ \ell t) + h(t) - \ell$ with $\reducedadmiss(0) = e$ then yields
\begin{equation}
    \admiss(t) = \translationpath(t) \circ \reducedadmiss(t).
\end{equation}
In particular as $\translationpath(1) = e$ we have $\admiss(1) = \reducedadmiss(1)$. Moreover the lift of $\reducedadmiss : [0,1] \rightarrow \Dmu^s(\T^n)$ with $\lift{\reducedadmiss}(0,x) = x$ satisfies
\begin{equation}
    \lift{\admiss}(t,x) = \translationpath(t) \circ \lift{\reducedadmiss}(t) = t\ell + \lift{\reducedadmiss}(t)
\end{equation}
and we have $\lift{\reducedadmiss}(1) = \lift{\admiss}(1) - \ell = \lift{\eta}$.

Lastly, we compare the actions
\begin{align*}
    2\big(\action{\admiss} - \action{\reducedadmiss} \big) &= \int_0^1 \bigl(\norm{\dot{\admiss}(t)}_{L^2}^2 - \norm{\dot{\reducedadmiss}(t)}_{L^2}^2\bigr) \, dt \\
    &= \int_0^1 \norm{v(t) + h(t)}_{L^2}^2 - \norm{v(t)\circ\translationpath(t) + h(t) - \ell}_{L^2}^2 \, dt.
\end{align*}
As mean-free fields are $L^2$-orthogonal to harmonic fields, and $\translationpath(t)$ is a volume-preserving diffeomorphism for all time, we have
\begin{align*}
    2 \big( \action{\admiss} - \action{\reducedadmiss} \big) &= \int_0^1 \Bigl(\norm{v(t)}_{L^2}^2 + \norm{h(t)}_{L^2}^2 - \norm{v(t)\circ\translationpath(t)}_{L^2}^2 - \norm{h(t) - \ell}_{L^2}^2\Bigr) \, dt \\
    &= 2\int_0^1 h(t) \cdot \ell \, dt - \abs{\ell}^2 \\
    &= \abs{\ell}^2 + 2c \cdot \ell \\
    &\geq 0
\end{align*}
with equality if and only if $\ell = 0$. Indeed, for every nonzero integer component $\ell_i$, the condition $\abs{c_i}<\frac{1}{2}$ gives
\begin{equation*}
    \ell_i^2+2c_i\ell_i>\ell_i^2-\abs{\ell_i}\geq0.
\end{equation*}
If $\ell=0$, the translation path is constant and $\reducedadmiss=\admiss$.
\end{proof}

We next record the following uniform bound on the pressure Hessian for radial geodesics in sufficiently small normal neighborhoods.

\begin{lemma}\label{pressure hessian bound}
For all $\delta>0$ there exists $\varepsilon>0$ such that, for all $u_0 \in B^s(0, \varepsilon)$, we have
\begin{equation}
\sup_{t \in [0,1]}\norm{\nabla^2 p(t)}_{L^\infty} \leq \delta
\end{equation}
where $p$ is the pressure in \eqref{euler} corresponding to $u_0$.
\end{lemma}

\begin{proof}
As we have
\begin{equation*}
-\Delta p = \mathrm{tr} (\nabla u)^2 \quad \text{and} \quad \int_{\T^n}p(x)\,dx = 0
\end{equation*}
along with $s>\frac{n}{2}+1$, the Sobolev lemma gives
\begin{equation*}
\norm{\nabla^2 p}_{L^\infty} \leq C \norm{\nabla^2 p}_{H^{s-1}} \leq C \norm{u}_{H^s}^2.
\end{equation*}
Using the standard energy estimate
\begin{equation*}
\frac{d}{dt}\norm{u(t)}_{H^s} \leq C \norm{u(t)}_{H^s}^2
\end{equation*}
we obtain
\begin{equation*}
\norm{u(t)}_{H^s} \leq \frac{\norm{u_0}_{H^s}}{1 - Ct\norm{u_0}_{H^s}}
\end{equation*}
for $0<t<\frac{1}{C\norm{u_0}_{H^s}}$. Hence, if $\varepsilon>0$ is sufficiently small, then
\begin{equation*}
\sup_{t \in [0,1]} \norm{\nabla^2 p(t)}_{L^\infty} \lesssim \varepsilon^2
\end{equation*}
and the result follows.
\end{proof}

Armed with the above, we proceed to the proof of the main result of the section.

\begin{proof}[Proof of Theorem \ref{local geodesic convexity}]
Let $\eta \in \normalhood{e}{\varepsilon}$ with $u_0 \in \normalball{\varepsilon}$ such that $\eta = \riemexp u_0$ and denote by $\geod : [0, 1] \rightarrow \normalhood{e}{\varepsilon}$ the radial geodesic connecting $e$ to $\eta$ with Eulerian velocity $u = \dot{\geod} \circ \geod^{-1}$ and mean-free pressure $p$.
As $L^2$-geodesics have constant speed we have that
\begin{equation*}
\action{\geod} = \frac{1}{2}\norm{u_0}_{L^2}^2.
\end{equation*}
We will show this is the minimum possible action for a path connecting $e$ to $\eta$. To this end, let $\admiss : [0,1] \rightarrow \Dmu^s$ be an admissible curve with $\admiss(0) = e$ and $\admiss(1) = \eta$.

For each $x\in \T^n$ consider the lifts of the paths $\geod(t,x)$ and $\admiss(t,x)$, denoted by $\lift{\geod}(t,x)$ and $\lift{\admiss}(t,x)$, to the universal cover $\R^n$ with $\lift{\geod}(0,x) = \lift{\admiss}(0,x) = x$, along with the periodic lift of the pressure which we continue to denote by $p$.

By smoothness of the $L^2$ exponential map in the $H^s$ topology and the Sobolev lemma, for sufficiently small $\varepsilon>0$, our lifted geodesic has unique decomposition $\lift{\geod}(t,x) = x + \zeta(t,x)$ with $\zeta(t,x) \in \big(-\frac{1}{2}, \frac{1}{2}\big)^n$ for all $(t,x) \in [0,1] \times \T^n$. The endpoint of the lift of the competitor, however, may differ by some integer translation $\lift{\admiss}(1) = \lift{\geod}(1) + \ell$ with $\ell \in \Z^n$. Nevertheless, after applying Lemma \ref{remove integer bulk}, we may assume without loss of generality that the lift of the path $\admiss$ has $\lift{\admiss}(1) = \lift{\geod}(1)$. 

Identifying $\T^n$ with the fundamental cell in the cover we examine the difference of the actions
\begin{align*}
    2 \big( \action{\admiss} - \action{\geod} \big) &= \int_0^1 \int_{\T^n} \bigl(\abs{\partial_t \lift{\admiss}}^2 - \abs{\partial_t \lift{\geod}}^2\bigr) \, dx \, dt \\
    &= \int_0^1 \int_{\T^n} \abs{\partial_t \lift{\admiss} - \partial_t \lift{\geod}}^2 \, dx \, dt + 2 \int_0^1 \int_{\T^n} \partial_t \lift{\geod} \cdot \big(\partial_t \lift{\admiss} - \partial_t \lift{\geod}\big) \, dx \, dt.
\end{align*}
We will prove that, for sufficiently small $\varepsilon>0$, the right hand side is positive unless $\sigma = \geod$.

For the first term note that $\lift{\admiss}(1) - \lift{\geod}(1) = \lift{\admiss}(0) - \lift{\geod}(0) = 0$ and hence, by the Poincar{\'e} inequality, we have
\begin{equation}
    \int_0^1 \abs{\partial_t \lift{\admiss}(t,x) - \partial_t \lift{\geod}(t,x)}^2 \, dt \geq \pi^2 \int_0^1 \abs{\lift{\admiss}(t,x) - \lift{\geod}(t,x)}^2 \, dt
\end{equation}
for each fixed $x \in \T^n$. Hence
\begin{equation}
    \int_0^1 \int_{\T^n} \abs{\partial_t \lift{\admiss} - \partial_t \lift{\geod}}^2 \, dx \, dt \geq \pi^2 \int_0^1 \int_{\T^n} \abs{\lift{\admiss} - \lift{\geod}}^2 \, dx \, dt.
\end{equation}

For the second term, using integration by parts in time and again that $\lift{\admiss} - \lift{\geod}$ vanishes at the endpoints, we have
\begin{equation}
    \int_0^1 \int_{\T^n} \partial_t \lift{\geod} \cdot \big(\partial_t \lift{\admiss} - \partial_t \lift{\geod}\big) \, dx \, dt = - \int_0^1 \int_{\T^n} \partial_t^2 \lift{\geod} \cdot \big(\lift{\admiss} - \lift{\geod}\big) \, dx \, dt = \int_0^1 \int_{\T^n} \big(\nabla p \circ \lift{\geod} \big) \cdot \big(\lift{\admiss} - \lift{\geod}\big) \, dx \, dt.
\end{equation}
Since both $\admiss(t)$ and $\geod(t)$ are volume-preserving,
\begin{equation}
    \int_{\T^n} p(t,\lift{\admiss}(t,x)) \, dx = \int_{\T^n} p(t,\lift{\geod}(t,x)) \, dx.
\end{equation}
Using Taylor's theorem in $\R^n$ along the straight segment joining $\lift{\geod}(t,x)$ and $\lift{\admiss}(t,x)$ we obtain
\begin{equation}
    \biggl|\int_0^1\int_{\T^n} \big( \nabla p \circ \lift{\geod} \big) \cdot \big(\lift{\admiss}-\lift{\geod}\big) \,dx\,dt \biggr| \leq \frac{1}{2} \sup_{t \in [0,1]}\norm{\nabla^2 p(t)}_{L^\infty} \int_0^1\int_{\T^n} \abs{\lift{\admiss}-\lift{\geod}}^2 \,dx\,dt
\end{equation}
and, by Lemma \ref{pressure hessian bound}, it follows that
\begin{equation}
    2 \big( \action{\admiss} - \action{\geod} \big) \geq \big( \pi^2 - C \varepsilon^2 \big) \int_0^1\int_{\T^n} \abs{\lift{\admiss}-\lift{\geod}}^2 \,dx\,dt.
\end{equation}

Hence, for $\pi^2 - C \varepsilon^2 >0$, the radial geodesic $\geod$ is the unique minimizer of the $L^2$ action functional among all admissible curves from $e$ to $\eta$.

The result for length minimization follows. Indeed any admissible curve with shorter $L^2$ length would, after reparameterization to constant speed, also have smaller $L^2$ action; a contradiction. Lastly, if an admissible curve has the same length as $\geod$, its constant-speed reparameterization has the same action and hence coincides with $\geod$ by uniqueness of the action minimizer.
\end{proof}

\begin{corollary}
There exists $\delta>0$ such that every pair $\eta,\varphi\in\normalhood{e}{\delta}$ is joined by an $L^2$ geodesic which is, up to reparameterization, the unique $L^2$ length-minimizer among all admissible curves in $\Dmu^s(\T^n)$ with the same endpoints.
\end{corollary}

\begin{proof}
Let $\varepsilon>0$ be as in Theorem~\ref{local geodesic convexity}. By continuity of composition and inversion in the $H^s$ topology, after choosing $\delta>0$ sufficiently small we have
\begin{equation*}
\eta,\varphi\in\normalhood{e}{\delta}\quad\Longrightarrow\quad\eta\circ\varphi^{-1}\in\normalhood{e}{\varepsilon}.
\end{equation*}
Set
\begin{equation*}
u=\riemexp^{-1}\bigl(\eta\circ\varphi^{-1}\bigr).
\end{equation*}
By right-invariance the curve
\begin{equation*}
\tilde{\geod}(t)=\riemexp(tu)\circ\varphi
\end{equation*}
is an $L^2$ geodesic from $\varphi$ to $\eta$. If $\alpha$ is any admissible curve from $\varphi$ to $\eta$, then $\alpha\circ\varphi^{-1}$ is an admissible curve from $e$ to $\eta\circ\varphi^{-1}$ and, by right-invariance,
\begin{equation*}
\length{\alpha\circ\varphi^{-1}}=\length{\alpha}.
\end{equation*}
Theorem~\ref{local geodesic convexity} therefore proves that $\tilde{\geod}$ is, up to reparameterization, the unique length-minimizer.
\end{proof}

The last result of this section establishes geodesic convexity of sufficiently small normal neighborhoods. To prove that the minimizing geodesic stays in the neighborhood we adapt the convexity argument of Whitehead \cite{whitehead1932convex}.

\begin{theorem}[Strong geodesic convexity]\label{geodesic convexity of normal neighborhoods}
Let $n\geq2$ and $s>\frac n2+1$. For sufficiently small $\varepsilon>0$, any two distinct fluid configurations $\xi,\varphi\in\closedhood{e}{\varepsilon}$ are joined by a geodesic segment whose interior lies in $\normalhood{e}{\varepsilon}$. This geodesic is, up to reparameterization, the unique $L^2$ length-minimizer among all admissible curves in $\Dmu^s(\T^n)$ with the same endpoints.
\end{theorem}

\begin{proof}
    Let $r>0$ be such that $\riemexp$ is a diffeomorphism on $\normalball{r}$ in the $H^s$ topologies and Theorem \ref{local geodesic convexity} holds on $\normalhood{e}{r}$. The inverse $\riemexp^{-1}$ then gives normal coordinates on $\normalhood{e}{r}$ in which the geodesic equation has the form
    \begin{equation}\label{geodesic equation in normal coordinates}
        \ddot \chi(t) = - \Gamma(\chi(t))\big(\dot \chi (t), \dot \chi(t)\big)
    \end{equation}
    with $\Gamma(v)$ a bounded bilinear map on $T_e\Dmu^s$ depending smoothly on $v$ in the $H^s$ operator norm. By continuity for any $C > \opnorm{\Gamma(0)}{H^s}$, after shrinking $r$ if necessary, we can ensure that both $Cr < 1$ and
    \begin{equation}\label{christoffel bound}
        \norm{\Gamma(x)(u,v)}_{H^s} \leq C \norm{u}_{H^s} \norm{v}_{H^s},
    \end{equation}
    for all $x \in \normalball{r}$ and all $u, v \in T_e\Dmu^s$. Next, using continuity of composition and inversion in the $H^s$ topologies, choose $0 < \varepsilon < \delta < r$ such that
    \begin{equation}\label{delta composition}
        \xi, \varphi \in \normalhood{e}{\delta} \quad \Longrightarrow \quad \xi \circ \varphi^{-1} \in \normalhood{e}{r}
    \end{equation}
    and
    \begin{equation}\label{varepsilon composition}
        \xi, \varphi \in \closedhood{e}{\varepsilon} \quad \Longrightarrow \quad \xi \circ \varphi^{-1} \in \normalhood{e}{\delta}.
    \end{equation}
    For distinct $\xi, \varphi \in \closedhood{e}{\varepsilon}$, let $u = \riemexp^{-1} (\xi \circ \varphi^{-1}) \in \normalball{\delta}$. By Theorem \ref{local geodesic convexity}, right invariance and \eqref{delta composition}-\eqref{varepsilon composition} above we have that $\gamma(t) = \riemexp(tu) \circ \varphi$ is the unique $L^2$ length-minimizing geodesic connecting $\varphi$ to $\xi$ in $\normalhood{e}{r}$. It remains to show that $\riemexp^{-1} \gamma(t) \in \normalball{\varepsilon}$ for all $0 < t < 1$. To this end, consider the function
    \begin{equation*}
        H: [0,1] \rightarrow \R \ ; \ t \mapsto \frac{1}{2}\norm{\riemexp^{-1} \gamma(t)}_{H^s}^2.
    \end{equation*}
    We will show that $H(t) < {\varepsilon^2}/{2}$ for all $0 < t < 1$. Letting $\chi(t) = \riemexp^{-1} \gamma(t)$ we have
    \begin{equation*}
        \ddot H(t) = \norm{\dot \chi(t)}_{H^s}^2 + \bigip{\ddot \chi(t), \chi(t)}_{H^s}.
    \end{equation*}
    By compactness of the unit interval, smoothness of $\riemexp^{-1}$ on $\normalhood{e}{r}$ and the fact that $\gamma(t)$ has constant $L^2$-speed, there exists $\rho > 0$ such that $\norm{\dot \chi(t)}_{H^s} \geq \rho$ for all $0 \leq t \leq 1$. Hence, from \eqref{geodesic equation in normal coordinates}-\eqref{christoffel bound}, we have
    \begin{equation*}
        \ddot H(t) \geq \norm{\dot \chi(t)}_{H^s}^2 - \norm{\ddot \chi(t)}_{H^s}\norm{\chi(t)}_{H^s} \geq \big(1-Cr\big)\rho^2 > 0
    \end{equation*}
    and consequently
    \begin{equation*}
        H(t) < (1-t)H(0) + t H(1) \leq \frac{\varepsilon^2}{2}.
    \end{equation*}
\end{proof}

\section{Distance comparison on normal neighborhoods}\label{local metric geometry}
Throughout this section we take $\varepsilon>0$ sufficiently small so that the $L^2$ Riemannian exponential map restricts to an $H^s$-diffeomorphism
\begin{equation*}
\riemexp:\normalball{\varepsilon}\subset T_e\Dmu^s \rightarrow \normalhood{e}{\varepsilon}\subset\Dmu^s.
\end{equation*}
We first study the local metric geometry of $\normalhood{e}{\varepsilon}$ through its embedding into the ambient linear space $L^2(\T^n,\R^n)$. Whenever necessary, we shrink $\varepsilon$ without changing notation. After such a shrinking, each $\eta\in\normalhood{e}{\varepsilon}$ has a unique canonical lift
\begin{equation*}
\lift{\eta}(x)=x+\zeta_\eta(x), \qquad \zeta_\eta\in H^s(\T^n,\R^n),
\end{equation*}
for which $\zeta_\eta(x)\in(-\tfrac{1}{2},\tfrac{1}{2})^n$ for every $x\in\T^n$. Among all lifts of $\eta$, this is the unique lift whose displacement from the identity has smallest $L^2$ norm. We identify $\eta$ with its periodic displacement
\begin{equation*}
\zeta_\eta=\lift{\eta}-\id \in H^s(\T^n,\R^n)\hookrightarrow L^2(\T^n,\R^n).
\end{equation*}
The ambient $L^2$ distance between $\eta,\xi\in\normalhood{e}{\varepsilon}$
is defined by
\begin{equation}\label{ambient linear distance}
\norm{\eta-\xi}_{L^2} = \norm{\lift{\eta}-\lift{\xi}}_{L^2} = \biggl( \int_{\T^n} \abs{\zeta_\eta(x)-\zeta_\xi(x)}^2\,dx \biggr)^{1/2}.
\end{equation}

\subsection{Local equivalence of distances}

The comparison of $\geodistance$, defined in \eqref{distance}, and the ambient linear distance \eqref{ambient linear distance} has been a topic of interest since the pioneering results of Shnirelman \cite{shnirelman1994generalized}, with recent improvements obtained by Schiffer and Zizza \cite{schiffer2024on}. These results say that on an $n$-dimensional cube, for $n\geq 3$, there exists $C>0$, depending only on the dimension, such that for all smooth\footnote{Here we use the notation $\Dmu$ for smooth volume-preserving diffeomorphisms and $\Dmu^s$ for those of Sobolev class $H^s$.} $\eta\in\Dmu([0,1]^n)$ with $\dist_{L^2}^{\Dmu}(e,\eta)<\infty$, one has
\begin{equation}\label{global Shnirelman inequality}
\norm{\eta - e}_{L^2} \leq \dist_{L^2}^{\Dmu}(e,\eta) \leq C\norm{\eta-e}_{L^2}.
\end{equation}
A global comparison of this type fails in the two-dimensional setting. Indeed, the ambient $L^2$ distance is uniformly bounded, whereas the volume-preserving diffeomorphism group of the square (or disk) equipped with the $L^2$ intrinsic distance has infinite diameter \cite{eliashberg1991the,shnirelman1994generalized}. However, in the discrete setting of permutations of square tilings, H\"older comparison estimates were established by Shnirelman and Zizza \cite{shnirelman1985the, zizza2025an}.

The purpose of this subsection is to prove a local analogue of \eqref{global Shnirelman inequality} on $H^s$-normal neighborhoods by alternative methods which, in contrast to the above also holds in the two-dimensional case.

\begin{theorem}[Local distance comparison]\label{local Shnirelman inequality}
Let $n\geq2$ and $s>\frac{n}{2}+1$. There exists $\varepsilon>0$ and $C(\varepsilon)>0$ such that for all $\eta\in \normalhood{e}{\varepsilon} \subset \Dmu^s(\T^n)$ one has
\begin{equation*}
\norm{\eta - e}_{L^2} \leq \geodistance(e,\eta) \leq C(\varepsilon)\norm{\eta-e}_{L^2}.
\end{equation*}
\end{theorem}

\begin{proof}
The lower bound is an elementary comparison once the lifts are fixed. Let $\eta\in\normalhood{e}{\varepsilon}$ have canonical lift $\lift{\eta}(x)=x+\zeta_\eta(x)$ with $\zeta_\eta(x)\in(-\tfrac12,\tfrac12)^n$, and let $\admiss:[0,1]\to\Dmu^s$ be any admissible path from $e$ to $\eta$. The lift of $\admiss$ beginning at the identity ends at $\lift{\eta}+\ell$ for some $\ell\in\Z^n$. Pointwise minimality of the canonical lift gives
\begin{equation*}
\norm{\eta-e}_{L^2}\leq
\norm{\lift{\eta}+\ell-e}_{L^2}.
\end{equation*}
Minkowski's inequality and volume preservation then imply
\begin{equation*}
\norm{\eta-e}_{L^2}
\leq \int_0^1\norm{\partial_t\lift{\admiss}(t)}_{L^2}\,dt
=\length{\admiss}.
\end{equation*}
Taking the infimum over all such paths yields the lower bound.

Now, let $\eta \in \normalhood{e}{\varepsilon}$ and $u_0 \in \normalball{\varepsilon}$ such that $\eta = \riemexp u_0$. From Theorem \ref{local geodesic convexity} we have that $\geodistance(e, \eta) = \norm{u_0}_{L^2}$. Hence it suffices to prove that there exists an $\varepsilon > 0$ and $C(\varepsilon)>0$ such that
\begin{equation*}
\norm{u_0}_{L^2} \leq C(\varepsilon)\norm{\eta-e}_{L^2}.
\end{equation*}

Letting $\geod(t) = \riemexp tu_0$ with corresponding Eulerian velocity $u(t)$, related by $\dot{\geod}(t) = u(t) \circ \geod(t)$ we have
\begin{align*}
\norm{\eta - e}_{L^2} &= \Bignorm{\int_0^1 \dot{\geod}(t)\,dt}_{L^2} \\
&= \Bignorm{\int_0^1 \big(u_0+\dot\geod(t)-u_0\big)\,dt}_{L^2} \\
&\geq \norm{u_0}_{L^2} - \int_0^1 \norm{\dot\geod(t)-u_0}_{L^2}\,dt.
\end{align*}

Using the Weber formula \eqref{Weber formula prelim}, the fact that $\geod(t)$ is volume-preserving, and that right-composition by $\geod(t)$ is an $L^2$-isometry, we obtain
\begin{equation*}
\begin{aligned}
    \norm{\dot\geod(t)-u_0}_{L^2} &= \norm{ \leray_e \Bigl((D\geod(t)^{-T}u_0)\circ\geod(t)^{-1}\Bigr) - u_0\circ\geod(t)^{-1}}_{L^2} \\
    &\leq\norm{\leray_e\Bigl(((D\geod(t)^{-T}-I)u_0)\circ\geod(t)^{-1}\Bigr)}_{L^2} + \norm{(\Id-\leray_e)(u_0\circ\geod(t)^{-1})}_{L^2}.
\end{aligned}
\end{equation*}

As the map $(t,u_0)\mapsto \riemexp tu_0$ is smooth, with respect to the $H^s$ topologies, from $[0,1]\times \normalball{\varepsilon}$ into $\Dmu^s$, for $\varepsilon>0$ sufficiently small, $\geod(t)$ remains uniformly close to $e$ in $H^s$ for $0\leq t\leq1$. Hence, by Sobolev embedding and the continuity of inversion in the $H^s$-topology, we have that the first term can be bounded by $c_1(\varepsilon)\norm{u_0}_{L^2}$ with $c_1(\varepsilon)\to0$ as $\varepsilon\to0$. For the second term we have the following.

\begin{lemma}\label{gradient projection bound}
For every $\eta\in\Dmu^s(\T^n)$ and $u\in T_e\Dmu^s(\T^n)$ we have
\begin{equation*}
\norm{(\Id-\leray_e)(u\circ\eta^{-1})}_{L^2} \leq C \norm{D\eta -\Id}_{L^\infty}\norm{u}_{L^2},
\end{equation*}
where
\begin{equation*}
\leray_e=\Id-\nabla\Delta^{-1}\div
\end{equation*}
is the Leray projection.
\end{lemma}

\begin{proof}
Let
\begin{equation*}
\gradproj=\Id-\leray_e = \nabla\Delta^{-1}\div.
\end{equation*}
Then
\begin{equation*}
\norm{\gradproj(u\circ\eta^{-1})}_{L^2} \leq C\norm{\div(u\circ\eta^{-1})}_{\dot H^{-1}}
\end{equation*}
for some positive constant $C$ depending only on dimension. Next, let $f\in H^1(\T^n)$ with $\norm{f}_{H^1}=1$. Changing variables by $\eta$ and integrating by parts, we obtain
\begin{align*}
\abs{\ip{\div (u \circ \eta^{-1}), f}_{L^2}} &= \abs{\ip{u \circ \eta^{-1}, \nabla f}_{L^2}} \\
&= \abs{\ip{u , \nabla f \circ \eta}_{L^2}} \\
&= \abs{\ip{u , \nabla f \circ \eta}_{L^2} - \ip{u, \nabla (f \circ \eta)}_{L^2}} \\
&= \abs{\ip{u, \bigl(\Id - D\eta^T\bigr)\nabla f \circ \eta}_{L^2}} \\
& \leq \norm{D\eta - \Id}_{L^\infty} \norm{u}_{L^2}
\end{align*}
which proves the claim.
\end{proof}

Consequently, by Sobolev lemma, the second term can be controlled by $c_2(\varepsilon)\norm{u_0}_{L^2}$ with $c_2(\varepsilon)\to0$ as $\varepsilon\to0$. Hence
\begin{equation*}
\norm{\dot\geod(t)-u_0}_{L^2} \leq c(\varepsilon)\norm{u_0}_{L^2},
\end{equation*}
where $c(\varepsilon) \to 0$ as $\varepsilon \to 0$ and we have
\begin{equation*}
\norm{\eta-e}_{L^2} \geq (1-c(\varepsilon))\norm{u_0}_{L^2}.
\end{equation*}
Choosing $\varepsilon$ so small that $1-c(\varepsilon)>0$, we get
\begin{equation*}
\norm{u_0}_{L^2} \leq C(\varepsilon)\norm{\eta-e}_{L^2}
\end{equation*}
and the theorem follows.
\end{proof}

By right-invariance, the preceding estimate immediately gives a two-point version after potentially shrinking the neighborhood.

\begin{corollary}\label{Phi Lipschitz}
There exists $\delta > 0$ and $C(\delta)>0$ such that for all $\eta,\varphi\in\normalhood{e}{\delta}$ we have
\begin{equation}\label{neighborhood Shnirelman inequality}
\norm{\eta-\varphi}_{L^2} \leq \geodistance(\eta,\varphi) \leq C(\delta)\norm{\eta-\varphi}_{L^2}.
\end{equation}
\end{corollary}

\begin{proof}
Let $\varepsilon>0$ and $C(\varepsilon)>0$ such that for all $\eta\in \normalhood{e}{\varepsilon}$ we have
\begin{equation*}
\norm{\eta - e}_{L^2} \leq \geodistance(e,\eta) \leq C(\varepsilon)\norm{\eta-e}_{L^2}.
\end{equation*}
By continuity of the group operations in the $H^s$ topology there exists $\delta>0$ such that
\begin{equation*}
\eta,\varphi\in\normalhood{e}{\delta} \quad \Longrightarrow \quad \eta\circ\varphi^{-1}\in\normalhood{e}{\varepsilon}.
\end{equation*}
Right-invariance of the $L^2$-geodesic distance then gives
\begin{equation*}
\geodistance(\eta,\varphi) = \geodistance(e,\eta\circ\varphi^{-1}).
\end{equation*}
Further, after shrinking $\delta$ so that all lifts are compatible, volume
preservation of $\varphi$ gives
\begin{equation*}
\norm{\eta\circ\varphi^{-1}-e}_{L^2} = \norm{\eta-\varphi}_{L^2}.
\end{equation*}
The result then follows from Theorem~\ref{local Shnirelman inequality}.
\end{proof}

\subsection{Compactness of closed normal neighborhoods}\label{compactness}
We prove that the images of sufficiently small $H^s$-closed balls around $0$ under the $L^2$-exponential map are compact with respect to the $L^2$-geodesic distance.

Throughout we will assume that $\varepsilon>0$ and $ \varepsilon'>0$ are such that the $L^2$-exponential map restricted to $\normalball{\varepsilon}$ and $\normalball[s']{\varepsilon'}$ is a diffeomorphism of the $H^s$- and $H^{s'}$-topologies respectively. We further assume that $\varepsilon>0$ is chosen sufficiently small so that $\normalball{\varepsilon} \subset \normalball[s']{\tfrac{\varepsilon'}{2}}$ and that
\begin{equation*}
    \closedhood{e}{\varepsilon} = \riemexp \closedball[s]{\varepsilon}
\end{equation*}
is contained in a slightly larger normal neighborhood on which both Theorem \ref{local geodesic convexity} and Corollary \ref{Phi Lipschitz} apply.

\begin{proposition}\label{compactness of closed normal neighborhood}
The set $\closedhood[s]{e}{\varepsilon}$ is compact with respect to the metric $\geodistance$. Moreover, $\closedhood[s]{e}{\varepsilon}$ is the $H^s$-closure of $\normalhood[s]{e}{\varepsilon}$.
\end{proposition}
\begin{proof}
Let $\{ \xi_j \} = \{ \riemexp u_j \} \subset \closedhood{e}{\varepsilon}$ with $\{ u_j \} \subset \closedball{\varepsilon}$. As $\{u_j\}$ is a bounded sequence in $H^s$, by weak compactness, after passing to a subsequence it converges weakly in $H^s$ to some $u\in T_e\Dmu^s$, with the divergence-free condition passing to the weak limit. The compact embedding $H^s\hookrightarrow H^{s'}$ gives a further subsequence which converges strongly in $H^{s'}$. The weak and strong limits agree because both subsequences converge to the same distribution. Weak lower semicontinuity gives
\begin{equation*}
\norm{u}_{H^s}\leq\varepsilon.
\end{equation*}
The $H^s$ and $H^{s'}$ Euler flows agree for initial velocities in $T_e\Dmu^s$, by uniqueness of solutions in the weaker class. The $H^{s'}$-continuity of the Riemannian exponential map therefore gives
\begin{equation*}
\xi_j=\riemexp u_j\rightarrow\riemexp u=\xi\quad\text{in }H^{s'}
\end{equation*}
and hence in $L^2$. Corollary~\ref{Phi Lipschitz} then gives
\begin{equation*}
\geodistance(\xi_j,\xi)\leq C\norm{\xi_j-\xi}_{L^2}\rightarrow0.
\end{equation*}

For the second assertion, let $u\in\closedball[s]{\varepsilon}$. Then $(1-\tfrac{1}{j})u\in\normalball[s]{\varepsilon}$ and
\begin{equation*}
\riemexp\big((1-\tfrac{1}{j})u\big)\rightarrow\riemexp u\quad\text{in }H^s
\end{equation*}
and hence $\closedhood{e}{\varepsilon}$ is contained in the $H^s$-closure of $\normalhood{e}{\varepsilon}$. For the reverse inclusion, suppose that $\{\xi_j\} \subset \normalhood{e}{\varepsilon}$ converges to some $\xi$ in the $H^s$ topology. As above, there exists a subsequence which converges to some limit in $\closedhood{e}{\varepsilon}$ with respect to the geodesic distance, and hence in the $L^2$ distance. As convergence in $H^s$ implies convergence in $L^2$, these limits agree, so $\xi \in \closedhood{e}{\varepsilon}.$
\end{proof}

\subsection{The second fundamental form and the distance comparison}
We end this section by explaining why the local comparison
\begin{equation*}
\geodistance(\eta,\varphi) \simeq \norm{\eta-\varphi}_{L^2}
\end{equation*}
is not a consequence of a bounded extrinsic curvature estimate. In finite dimensions, bounded second fundamental form does force local chord-arc estimates. We record the elementary finite-dimensional model.

\begin{lemma}\label{chord arc estimate}
Let $\Sigma \subset \R^N$ be a smooth embedded submanifold with the induced metric and exponential map $\exp^\Sigma$. Let $p \in \Sigma$. If there exists a normal neighborhood $\U(p, \varepsilon)$ around $p$ and a constant $\kappa>0$ such that
\begin{equation}\label{kappa bound}
\kappa\,\dist^\Sigma(p,q) < \pi \quad \text{ for all } q \in \U(p, \varepsilon)
\end{equation}
and
\begin{equation}\label{II quadratic bound}
\abs{II_{\geod(t)}(\dot{\geod}(t),\dot{\geod}(t))} \leq \kappa \abs{\dot{\geod}(t)}^2 \quad \text{ for all } v \in B(0, \varepsilon),\ t \in [0,1]
\end{equation}
where $\geod(t) = \exp^\Sigma_p(tv)$ and $II$ is the second fundamental form for $\Sigma \subset \R^N$, then there exists a constant $C>0$ such that
\begin{equation}\label{finite dimensional local ratio}
\dist^\Sigma(p,q) \leq C \abs{p-q}  \quad \text{ for all } q \in \U(p, \varepsilon)
\end{equation}
where $\dist^\Sigma$ is the induced distance on $\Sigma \subset \R^N$.
\end{lemma}

\enlargethispage{2\baselineskip}

\begin{proof}
Let $q \in \U(p, \varepsilon) \setminus \{p\}$ with $\geod : [0,1] \rightarrow \U(p, \varepsilon)$ the corresponding radial geodesic with $\geod(0) = p$, $\geod(1) = q$ and $\abs{\dot{\geod}(t)} = \dist^\Sigma(p,q)$ for all $t \in [0,1]$. By assumption, we have that
\begin{equation}
\abs{\ddot{\geod}(t)} = \abs{II_{\geod(t)}(\dot{\geod}(t),\dot{\geod}(t))} \leq \kappa \abs{\dot{\geod}(t)}^2 = \kappa \dist^\Sigma(p,q)^2 \quad \text{ for all } t \in [0,1].
\end{equation}
Next define the unit tangent $X : [0,1] \rightarrow T\Sigma \ ; \ X(t) = \frac{\dot{\geod}(t)}{\dist^\Sigma(p,q)}$. So
\begin{equation*}
\abs{X(t)} = 1, \quad \abs{\dot{X}(t)} \leq \kappa \dist^\Sigma(p,q) \quad \text{ for all } t\in[0,1].
\end{equation*}
In particular $X(t)$ evolves in the unit sphere of each tangent space $T_{\geod(t)}\Sigma$. Using \eqref{kappa bound} and \eqref{II quadratic bound}, the angle between $X(t)$ and $X(\tfrac{1}{2})$ (after identifying tangent spaces in $\R^N$) can be bounded as
\begin{equation}\label{angle bound}
0 \leq \theta(X(t), X(\tfrac{1}{2})) \leq \biggl|\int_{\tfrac{1}{2}}^t \abs{\dot{X}(r)}\,dr\biggr| \leq \kappa \dist^\Sigma(p,q) \abs{t -\tfrac{1}{2}} < \frac{\pi}{2}.
\end{equation}
Hence we have
\begin{equation*}
q-p = \int_0^1 \dot{\geod}(t)\,dt = \dist^\Sigma(p,q) \int_0^1 X(t)\,dt
\end{equation*}
and we can estimate
\begin{align*}
\abs{p-q} &= \dist^\Sigma(p,q) \biggl\vert \int_0^1 X(t)\,dt \biggr\vert \\
&\geq \dist^\Sigma(p,q) \bigip{\int_0^1X(t)\,dt, X(\tfrac{1}{2})} \\
& = \dist^\Sigma(p,q) \int_0^1 \bigip{X(t), X(\tfrac{1}{2})}\,dt \\
&\geq \dist^\Sigma(p,q) \int_0^1 \cos(\theta(X(t), X(\tfrac{1}{2})))\,dt.
\end{align*}
As cosine is decreasing on $[0,\tfrac{\pi}{2}]$, we now, from \eqref{angle bound}, have
\begin{equation*}
\abs{p-q} \geq \dist^\Sigma(p,q) \int_0^1 \cos(\kappa \dist^\Sigma(p,q) \abs{t -\tfrac{1}{2}})\,dt = \frac{2}{\kappa} \sin\biggl(\frac{\kappa \dist^\Sigma(p,q)}{2}\biggr).
\end{equation*}
As $\sin(x) \geq \frac{x}{2}$ on $[0 , \tfrac{\pi}{2}]$, we have
\begin{equation*}
\abs{p-q} \geq \frac{\dist^\Sigma(p,q)}{2} \quad \text{for all } q \in \U(p, \varepsilon).
\end{equation*}
\end{proof}

In the hydrodynamical setting, the pressure term is the formal analogue of the normal acceleration in the geodesic equation. Namely, the Lagrangian form of the Euler equation is
\begin{equation*}
\ddot\geod(t)=-\nabla p(t)\circ\geod(t).
\end{equation*}
Hence $-\nabla p\circ\geod$ plays the role of the second fundamental form in this $L^2$ setting. Since $\Dmu^s$ is not an $L^2$ Hilbert submanifold of the ambient space, we use this language only as an analogy and formulate the next statement directly in terms of the pressure-gradient quadratic map.

\begin{proposition}\label{local ratio unbounded}
Let $\varepsilon>0$. For all $\kappa>0$, there exists $u \in \normalball{\varepsilon} \setminus \{0\} \subset T_e\Dmu^s(\T^n)$ such that, for $\geod(t) = \riemexp tu$ we have
\begin{equation}\label{Euler II unbounded}
\norm{\nabla p(1)\circ\geod(1)}_{L^2}
\geq \kappa \norm{\dot{\geod}(1)}_{L^2}^2
\end{equation}
where $p$ is the pressure in \eqref{euler}. Equivalently, the pressure-gradient
quadratic map is not locally bounded from $L^2\times L^2$ to $L^2$.
\end{proposition}

\begin{remark}
In particular, the pressure-gradient map is not bounded in the $L^2$ topology in the manner required by Lemma~\ref{chord arc estimate}. This result does not indicate a pathology in the $H^s$ manifold structure. The $L^2$ Riemannian exponential map is smooth in the $H^s$ topology, and the pressure-gradient quadratic map satisfies the corresponding $H^s$ estimates on bounded sets. The unboundedness occurs only when both inputs and the output are measured in $L^2$.
\end{remark}

\begin{proof}
    We construct the sequence explicitly on the two-dimensional torus $\T^2$. The $n$-dimensional case follows by trivial extension. We begin with the Laplace eigenfunctions
    \begin{equation*}
    \psi_k(x,y) = -\frac{1}{2 \pi k}\sin(2 \pi kx)\sin(2 \pi ky), \qquad -\Delta\psi_k = 8 \pi^2 k^2\psi_k, \qquad k\in\N,
    \end{equation*}
    whose skew-gradients are stationary Euler fields. Setting $\alpha_k= \frac{\varepsilon}{\sqrt{2(1 + (8\pi^2k^2)^s)}}$ we define
    \begin{equation*}
    u_k(x,y) = \alpha_k \symp \psi_k(x,y) = \alpha_k\Big( \sin(2 \pi kx)\cos(2 \pi ky)\,\partial_x -\sin(2 \pi ky)\cos(2 \pi kx)\,\partial_y \Big),
    \end{equation*}
    whose $L^2$ and $H^s$ norms are given by
    \begin{equation*}
    \norm{u_k}_{L^2} =\frac{\alpha_k}{\sqrt2} =\frac{\varepsilon}{2\sqrt{1 + (8\pi^2k^2)^s}} \quad \text{and} \quad  \norm{u_k}_{H^s} =\sqrt{1 + (8\pi^2k^2)^s}\,\norm{u_k}_{L^2} =\frac{\varepsilon}{2}.
    \end{equation*}
    In particular, $u_k\in \normalball{\varepsilon}$ for all $k \in \N$. Applying the covariant derivative, we have
    \begin{equation*}
    \nabla u_k(x,y)
     =\alpha_k 2 \pi k
     \begin{bmatrix}
     \cos(2 \pi kx)\cos(2 \pi ky) & -\sin(2 \pi kx)\sin(2 \pi ky)\\
     \sin(2 \pi kx)\sin(2 \pi ky) & -\cos(2 \pi kx)\cos(2 \pi ky)
     \end{bmatrix}.
    \end{equation*}
    A direct calculation gives the pressure gradient
    \begin{equation*}
    \nabla p_k =-\nabla_{u_k}u_k =-\alpha_k^2 2 \pi k\Big( \sin(2 \pi kx)\cos(2 \pi kx)\,\partial_x +\sin(2 \pi ky)\cos(2 \pi ky)\,\partial_y \Big),
    \end{equation*}
    and hence
    \begin{equation*}
    \norm{\nabla p_k}_{L^2} = \alpha_k^2 \pi k =\frac{\varepsilon^2 \pi k}{2\big(1+(8 \pi^2 k^2)^s\big)}.
    \end{equation*}
    Therefore
    \begin{equation*}
    \frac{\norm{\nabla p_k}_{L^2}}{\norm{u_k}_{L^2}^2}=2 \pi k \rightarrow\infty
    \end{equation*}
    and the proposition follows.
\end{proof}

\begin{remark}
Hence the local comparison of intrinsic and ambient $L^2$ distances in Theorem~\ref{local Shnirelman inequality} is not explained by a naive finite-dimensional bounded-second-fundamental-form argument. The pressure term is unbounded as a quadratic map in $L^2$. The local comparison made here instead relies on the coadjoint conservation law.
\end{remark}

\section{Metric betweenness and the hydrodynamical Hopf-Rinow theorem}\label{Hopf-Rinow}

A central question in geometric hydrodynamics is the two-point boundary-value problem: given two attainable fluid configurations (defined below), does there exist an Euler flow connecting them? As mentioned in the introduction, we distinguish two versions of this problem. The weak version asks only for an $L^2$ geodesic with the prescribed endpoints, or equivalently for surjectivity of the $L^2$-Riemannian exponential map onto the attainable component, while the strong version asks for an $L^2$-length-minimizing geodesic. The construction of the smooth exponential map, together with the inverse function theorem, gives a local solution to the weak problem, while Theorem \ref{local geodesic convexity} establishes the corresponding solution of the strong problem in a sufficiently small normal neighborhood. In this section we formulate, for the two-dimensional setting, a criterion for the strong problem in terms of Menger's notion of metric betweenness \cite{menger1928untersuchungen}.

As outlined in the preliminaries, we say that $\eta \in \Dmu^s$ is attainable from the identity if there exists an admissible curve connecting $e$ to $\eta$ or, equivalently, if $\geodistance(e, \eta) < \infty$. For an attainable $\eta$ we define its corresponding excess functional, which, for $\xi \in \Dmu^s$ attainable from $e$, is given by
\begin{equation}\label{eta excess}
    E_\eta(\xi) = \geodistance(e, \xi) + \geodistance(\xi, \eta) - \geodistance(e, \eta).
\end{equation}
Notice that the triangle inequality gives $E_\eta(\xi) \geq 0$. We say that $\xi$ is metrically between $e$ and $\eta$ if $E_\eta(\xi) = 0$, i.e., if
\begin{equation*}
    \geodistance(e, \eta) = \geodistance(e, \xi) + \geodistance(\xi, \eta)
\end{equation*}
and further refer to such $\xi$ as being nontrivial if it is distinct from both $e$ and $\eta$.

Our central result of this section is the following.

\begin{theorem}[Metric betweenness criterion]\label{Hopf-Rinow criteria}
    Let $s>2$ and $\varepsilon>0$ sufficiently small so that Theorem \ref{local geodesic convexity} applies on a $H^s$ ball of radius larger than $\varepsilon$. The following are equivalent.
    \begin{enumerate}
        \item[(i)] For all $\eta \in \Dmu^s(\T^2) \setminus \closedhood{e}{\varepsilon}$ attainable from the identity, there exists $\xi \in \closedhood{e}{\varepsilon} \setminus \{e\}$ which is metrically between $e$ and $\eta$. \\
        \item[(ii)] For all $\eta \in \Dmu^s(\T^2) \setminus \closedhood{e}{\varepsilon}$  attainable from the identity, there exists an $L^2$-length-minimizing geodesic connecting $e$ and $\eta$.
    \end{enumerate}
\end{theorem}

\begin{remark}
    When paired with Theorem \ref{local geodesic convexity}, the above gives an equivalent criterion for establishing or disproving the strong hydrodynamical Hopf-Rinow Theorem in the two-dimensional case.

    A related metric construction appears in \cite[Theorem~2.4.16]{burago2001course}: a complete metric space in which every pair of points has a midpoint is a geodesic space. However, the proof uses metric completeness which is not available in our setting. Here the metrically between point need not be a midpoint, and the argument instead uses geodesic convexity and completeness of the two-dimensional Euler flow, and uniqueness of the local minimizing geodesic.
\end{remark}

The restriction to the two-dimensional case is crucial for the argument on account of the use of geodesic completeness or, in other words, the global existence of solutions (a standard assumption in the finite dimensional Hopf-Rinow theorem). Furthermore, every point on a length-minimizing geodesic is metrically between its endpoints, whereas the same cannot be said of a point along a nonminimizing geodesic.

\begin{figure}[H]
\centering
\begin{tikzpicture}[
    x=1cm,
    y=1cm,
    >=stealth,
    sphere/.style={thick,fill=black!3},
    hidden/.style={gray,densely dashed,line width=.85pt},
    visible/.style={gray,line width=.85pt},
    geodesic/.style={blue,line width=1.15pt},
    minimizing/.style={ForestGreen!75!black,line width=1.25pt},
    point/.style={circle,fill=black,inner sep=1.7pt}
]

\begin{scope}[xshift=-3.35cm]
\draw[sphere] (0,0) circle (1.50);
\draw[hidden]
    (1.50,0)
    arc[start angle=0,end angle=180,
        x radius=1.50,y radius=.33];
\draw[visible]
    (-1.50,0)
    arc[start angle=180,end angle=360,
        x radius=1.50,y radius=.33];

\coordinate (Nleft) at (0,1.50);
\coordinate (xleft) at
    ({.55*sin(52)},{1.50*cos(52)});
\coordinate (pleft) at
    ({.55*sin(90)},{1.50*cos(90)});

\draw[geodesic,->]
    plot[
        domain=0:90,
        samples=70,
        variable=\t
    ]
    ({.55*sin(\t)},{1.50*cos(\t)});

\node[point,label=above:$N$] at (Nleft) {};
\node[point,label=right:$p$] at (pleft) {};

\fill[white] (xleft) circle (2pt);
\draw[black,line width=.85pt] (xleft) circle (2pt);
\node[right] at
    ({.55*sin(52)+.09},{1.50*cos(52)+.04}) {$x$};

\node[font=\small] at (1.1,1.7)
    {$\quad \mathbb{S}^2\setminus\{x\}$};
\end{scope}

\begin{scope}[xshift=3.35cm]
\draw[sphere] (0,0) circle (1.50);
\draw[hidden]
    (1.50,0)
    arc[start angle=0,end angle=180,
        x radius=1.50,y radius=.33];
\draw[visible]
    (-1.50,0)
    arc[start angle=180,end angle=360,
        x radius=1.50,y radius=.33];

\coordinate (Nright) at (0,1.50);
\coordinate (Sright) at (0,-1.50);
\coordinate (pright) at
    ({.42*sin(235)},{1.50*cos(235)});

\draw[geodesic,densely dashed]
    plot[
        domain=0:180,
        samples=80,
        variable=\t
    ]
    ({.42*sin(\t)},{1.50*cos(\t)});

\draw[geodesic,->]
    plot[
        domain=180:235,
        samples=35,
        variable=\t
    ]
    ({.42*sin(\t)},{1.50*cos(\t)});

\draw[minimizing,->]
    plot[
        domain=360:235,
        samples=70,
        variable=\t
    ]
    ({.42*sin(\t)},{1.50*cos(\t)});

\node[point,label=above:$N$] at (Nright) {};
\node[point,label=below right:$S$] at (Sright) {};
\node[point,label=left:$p$] at (pright) {};

\node[font=\small] at (0.9,1.7)
    {$\quad \mathbb{S}^2$};
\end{scope}

\end{tikzpicture}
\caption{The image on the left shows the importance of completeness. Indeed, there are nontrivial points which are metrically between the north pole $N$ and the point $p$, but there exists no length-minimizing geodesic connecting $N$ to $p$. On the right we see the natural relationship between metric betweenness and minimizing geodesics: each point along the green geodesic is metrically between $N$ and $p$, while none of the interior points along the blue geodesic are metrically between $N$ and $p$.}
\label{sphere metric betweenness figure}
\end{figure}
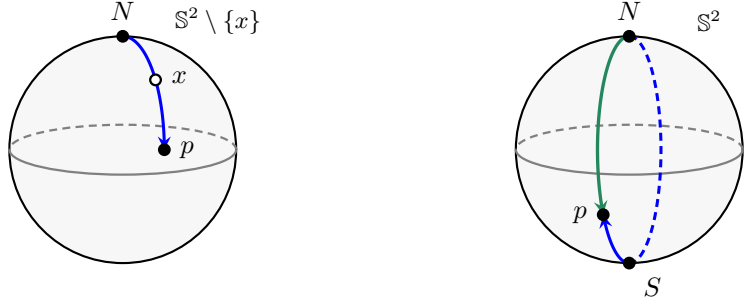

\begin{proof}[Proof of Theorem \ref{Hopf-Rinow criteria}]
    Assume that $(i)$ holds and let $\eta \in \Dmu^s \setminus \closedhood{e}{\varepsilon}$ be attainable from the identity. By assumption, there exists $\xi \in \closedhood{e}{\varepsilon} \setminus \{e\}$ which is metrically between $e$ and $\eta$. Let $u_\xi = \riemexp^{-1} \xi \in \closedball{\varepsilon} \setminus \{0\}$ and $\hat u_\xi = \frac{u_\xi}{\norm{u_\xi}_{L^2}}$. The geodesic given by $\geod(t) = \riemexp t \hat u_\xi$ is defined for all $t \geq 0$, has unit $L^2$-speed and, by Theorem \ref{local geodesic convexity}, is $L^2$-length-minimizing for $0 \leq t \leq \norm{u_\xi}_{L^2}$. Furthermore, using the fact that $\xi$ is metrically between $e$ and $\eta$, the length-minimizing property gives that $\geod(t)$ is metrically between $e$ and $\eta$ for all $0 \leq t \leq \norm{u_\xi}_{L^2}$. Next let
    \begin{equation*}
        \mathcal{S}=\bigl\{\tau\in\big[0,\geodistance(e, \eta)\big]:\geodistance(\geod(\tau),\eta)=\geodistance(e, \eta)-\tau\bigr\}.
    \end{equation*}
    If $\tau\in\mathcal{S}$ and $0\leq t\leq\tau$, the triangle inequality and the unit speed of $\geod$ give
    \begin{equation*}
        \geodistance(e, \eta)\leq\geodistance(e,\geod(t))+\geodistance(\geod(t),\eta)
        \leq t+(\tau-t)+(\geodistance(e, \eta)-\tau)=\geodistance(e, \eta).
    \end{equation*}
    Consequently $\geodistance(e,\geod(t))=t$ and $\geodistance(\geod(t),\eta)=\geodistance(e, \eta)-t$. Hence $\mathcal{S}$ is an interval, and for every $\tau\in\mathcal{S}$ the geodesic $\geod$ is $L^2$ length-minimizing on $[0,\tau]$, with $\geod(t)$ metrically between $e$ and $\eta$ for every $t\in[0,\tau]$. Furthermore, as $\geodistance(\geod(t),\geod(t'))\leq\abs{t-t'}$, we have that $\mathcal{S}$ is closed. Set $T = \sup \mathcal{S}$.  From the discussion above we have that $T \in \mathcal{S}$ and $[0, \norm{u_\xi}_{L^2}] \subset \mathcal{S}$, so $T>0$. Assume that $T < \geodistance(e, \eta)$ and let $\varphi = \geod(T) \neq \eta$. Right-invariance gives
    \begin{equation*}
        \geodistance(e, \eta \circ \varphi^{-1}) = \geodistance(\varphi, \eta) < \geodistance(e, \eta)
    \end{equation*}
    and hence $\eta \circ \varphi^{-1}$ is attainable from the identity. If $\eta \circ \varphi^{-1} \in \closedhood{e}{\varepsilon}$, let $\zeta = \eta \circ \varphi^{-1}$, otherwise by the assumption $(i)$, we have that there exists $\zeta \in \closedhood{e}{\varepsilon} \setminus \{e\}$ which is metrically between $e$ and $\eta \circ \varphi^{-1}$. Let $u_\zeta = \riemexp^{-1} \zeta \in \closedball{\varepsilon} \setminus \{0\}$ and $\hat u_\zeta = \frac{u_\zeta}{\norm{u_\zeta}_{L^2}}$. As above, the geodesic given by $t \mapsto \riemexp t \hat u_\zeta$ is defined for all $t \geq 0$, has unit $L^2$-speed and, for $0 \leq t \leq \norm{u_\zeta}_{L^2}$, is both $L^2$-length-minimizing and metrically between $e$ and $\eta \circ \varphi^{-1}$. Right-invariance then gives the geodesic $\tilde \geod(t) = \riemexp (t \hat u_\zeta) \circ \varphi$ is both $L^2$ length-minimizing and metrically between $\varphi$ and $\eta$ for all $0 \leq t \leq \norm{u_\zeta}_{L^2}$. Hence $\zeta \circ \varphi$ is metrically between $\varphi$ and $\eta$. In fact, $\zeta \circ \varphi$ is metrically between $e$ and $\eta$. Indeed, from the triangle inequality it follows that
    \begin{align*}
        \geodistance(e, \zeta \circ \varphi) + \geodistance(\zeta \circ \varphi, \eta) &\leq \geodistance(e, \varphi) + \geodistance(\varphi, \zeta \circ \varphi) + \geodistance(\zeta \circ \varphi, \eta) \\
        &= \geodistance(e, \varphi) + \geodistance(\varphi, \eta) \\
        &= \geodistance(e, \eta) \\
        &\leq \geodistance(e, \zeta \circ \varphi) + \geodistance(\zeta \circ \varphi, \eta),
    \end{align*}
    so every inequality is in fact an equality. Consequently we have
    \begin{align*}
        \geodistance(e, \zeta \circ \varphi) &= \geodistance(e, \eta) - \geodistance(\zeta \circ \varphi, \eta) \\
        &= \geodistance(e, \varphi) + \geodistance(\varphi, \eta) - \geodistance(\zeta \circ \varphi, \eta) \\
        &= \geodistance(e, \varphi) + \geodistance(\varphi, \zeta \circ \varphi) \\
        & = \length{\geod\vert_{[0, T]}} + \length{\tilde{\geod}\vert_{[0, \norm{u_\zeta}_{L^2}]}}.
    \end{align*}
    and hence the concatenated curve $\geod \sqcup \tilde{\geod}$ is a length-minimizing admissible curve from $e$ to $\zeta \circ \varphi$. The standard variational no-corner argument gives that the tangent vectors to both geodesics agree at the point of concatenation. Indeed, take two sufficiently close points on either side of the join at $\varphi$. After right-translation by $\varphi^{-1}$, these points lie in a normal neighborhood of the identity on which Theorem~\ref{geodesic convexity of normal neighborhoods} applies. The translated portion of the concatenation is itself length-minimizing, and hence coincides, up to reparameterization, with the unique minimizing geodesic between these points. Both pieces have unit $L^2$ speed, so their tangent vectors agree at the join. Hence $\tilde \gamma(t) = \gamma(T + t)$ for $0 \leq t \leq \norm{u_\zeta}_{L^2}$. Consequently $\geod$ is length-minimizing and metrically between $e$ and $\eta$ for all $0 \leq t \leq T + \norm{u_\zeta}_{L^2}$. In particular $T + \norm{u_\zeta}_{L^2} \in \mathcal{S}$; a contradiction. So $\geod : [0, \geodistance(e, \eta)] \rightarrow \Dmu^s$ is a length-minimizing geodesic from $e$ to $\eta$.

    Assume next that $(ii)$ holds and let $\eta \in \Dmu^s \setminus \closedhood{e}{\varepsilon}$ be attainable from the identity. By assumption there exists $u \in T_e\Dmu^s$ such that $\gamma: [0,1] \rightarrow \Dmu^s \ ; \  t \mapsto \riemexp t u$, is a length-minimizing geodesic connecting $e$ to $\eta$. Let $t_\varepsilon \in (0,1)$ such that $t_\varepsilon u \in \normalball{\varepsilon}$. Every point in the interior of a length-minimizing geodesic is metrically between $e$ and $\eta$. Letting $\xi = \riemexp t_\varepsilon u$ yields the result.
\end{proof}

\enlargethispage{3\baselineskip}

A natural initial attempt to construct a nontrivial $\xi \in \closedhood{e}{\varepsilon}$ which is metrically between $e$ and $\eta$ is to take a sequence of admissible curves connecting them whose length tends towards $\geodistance(e, \eta)$. One then considers the sequence of `first exit points' where the admissible curve leaves the normal neighborhood (a necessary condition if $\eta$ lies in the complement). The following proposition explores this approach and leads us to the key analytical obstacle. As in Subsection \ref{compactness} we assume that $\varepsilon>0$ and $ \varepsilon'>0$ are such that the $L^2$-exponential map restricted to $\normalball{\varepsilon}$ and $\normalball[s']{\varepsilon'}$ is a diffeomorphism of the $H^s$- and $H^{s'}$-topologies respectively. We further assume that $\varepsilon>0$ is chosen sufficiently small so that $\normalball{\varepsilon} \subset \normalball[s']{\tfrac{\varepsilon'}{2}}$ and that
\begin{equation*}
    \closedhood{e}{\varepsilon} = \riemexp \closedball[s]{\varepsilon}
\end{equation*}
is contained in a slightly larger normal neighborhood on which both Theorem \ref{local geodesic convexity} and Corollary \ref{Phi Lipschitz} apply.

\begin{proposition}
    For all $\eta \in \Dmu^s(\T^n) \setminus \closedhood{e}{\varepsilon}$ attainable from the identity, there exists a sequence $\{\xi_j\} \subset \closedhood{e}{\varepsilon}$ and a point $\xi \in \closedhood{e}{\varepsilon}$ such that
    \begin{equation*}
        \norm{\riemexp^{-1} \xi_j}_{H^s} = \varepsilon, \quad \geodistance(\xi, \xi_j) \to 0
    \end{equation*}
    and $\xi$ is metrically between $e$ and $\eta$.
\end{proposition}

\begin{proof}
    Consider a sequence of admissible curves $\sigma_j$ connecting $e$ to $\eta$ with $\length{\sigma_j} < \geodistance(e, \eta) + \frac{1}{j}$ and let $\xi_j$ be the first point where $\sigma_j(t) \notin \normalhood{e}{\varepsilon}$. As each $\sigma_j$ is continuous in the $H^s$ topology we have $\xi_j \in \overline{\normalhood{e}{\varepsilon}}^{_{H^s}} \setminus \normalhood{e}{\varepsilon} = \closedhood{e}{\varepsilon} \setminus \normalhood{e}{\varepsilon}$ and consequently $\norm{\riemexp^{-1} \xi_j}_{H^s} = \varepsilon$.

    By Proposition \ref{compactness of closed normal neighborhood} there exists a subsequence, which we continue to denote by $\{\xi_j\}$, and some $\xi \in \closedhood{e}{\varepsilon}$ for which $\geodistance(\xi, \xi_j) \to 0$. Moreover, $\xi_j \to \xi$ strongly in the $H^{s'}$ topology.

    It remains to show that $\xi$ is metrically between $e$ and $\eta$. Expanding the excess functional we have
    \begin{align*}
        E_\eta(\xi) &= E_\eta(\xi) - E_\eta(\xi_j) + E_\eta(\xi_j) \\
        & = \big(\geodistance(e, \xi) - \geodistance(e, \xi_j)\big) + \big(\geodistance(\xi, \eta) - \geodistance(\xi_j, \eta)\big)  \\
        &  \qquad \qquad + \big(\geodistance(e, \xi_j) + \geodistance(\xi_j, \eta) - \geodistance(e, \eta)\big).
    \end{align*}
    Applying the triangle inequality to the first two terms and recalling that $\length{\sigma_j} < \geodistance(e, \eta) + \frac{1}{j}$ then gives
    \begin{equation*}
        E_\eta(\xi) \leq 2\geodistance(\xi, \xi_j) + \frac{1}{j}
    \end{equation*}
    and the result follows.
\end{proof}

\begin{remark}\label{finite dimensional zero excess mechanism}
The role of metric betweenness in the classical Hopf-Rinow argument can be seen from the compactness of small metric spheres. Let $(M,g)$ be a finite-dimensional Riemannian manifold, let $p,q\in M$, and choose
\begin{equation*}
    0<r<d(p,q)
\end{equation*}
sufficiently small that $\overline{B}(p,r)$ is contained in a normal neighborhood of $p$. The metric sphere
\begin{equation*}
    \partial B(p, r)=\{x\in M:d(p,x)=r\}
\end{equation*}
is then compact. Let $\sigma_j$ be piecewise $C^1$ curves from $p$ to $q$ whose lengths converge to $d(p,q)$, and let $x_j$ be the first point at which $\sigma_j$ meets $\partial B(p, r)$. We have
\begin{equation*}
    d(p,q)\leq r+d(x_j,q)\leq\mathcal L(\sigma_j)\rightarrow d(p,q).
\end{equation*}
After passing to a subsequence, compactness gives
\begin{equation*}
    x_j\rightarrow x\in \partial B(p, r).
\end{equation*}
Continuity of the distance therefore gives
\begin{equation*}
    d(p,q)=d(p,x)+d(x,q)=r+d(x,q).
\end{equation*}
Moreover, for every $z\in\overline{B}(p,r)$, the triangle inequality gives
\begin{equation*}
    d(z,q)\geq d(p,q)-d(p,z)\geq d(p,q)-r.
\end{equation*}
Consequently,
\begin{equation*}
    \min_{z\in\overline{B}(p,r)}d(z,q) =\min_{z\in \partial B(p, r)}d(z,q) =d(p,q)-r,
\end{equation*}
and the minimizers over $\overline{B}(p,r)$ are precisely the points of $\partial B(p, r)$ which are metrically between $p$ and $q$. Hence, in finite dimensions, minimizing the distance to the target over a sufficiently small closed metric ball automatically produces the intermediate point required by the continuation argument. Together with geodesic completeness and the no-corner argument, this is the mechanism underlying the classical Hopf-Rinow theorem.

There are two distinctions in the hydrodynamical setting. First, the corresponding first-exit set is
\begin{equation*}
    \partial \normalhood{e}{\varepsilon} = \closedhood{e}{\varepsilon}\setminus\normalhood{e}{\varepsilon} = \left\{\riemexp u:\norm{u}_{H^s}=\varepsilon\right\},
\end{equation*}
whereas Theorem \ref{local geodesic convexity} gives
\begin{equation*}
    \geodistance(e,\riemexp u)=\norm{u}_{L^2}.
\end{equation*}
Hence $\partial \normalhood{e}{\varepsilon}$ is not a sphere for the geodesic distance. In particular, minimizing $\geodistance(\xi,q)$ over $\partial \normalhood{e}{\varepsilon}$ is not equivalent to requiring
\begin{equation*}
    \geodistance(e,q) = \geodistance(e,\xi)+\geodistance(\xi,q).
\end{equation*}
The weak metric acts as a distorting lens on the normal neighborhood: a point may be closer to the target without being metrically between the identity and the target. Hence metric betweenness is the appropriate notion in this context.

Second, $\partial \normalhood{e}{\varepsilon}$ is not closed, and hence is not compact, with respect to the geodesic distance. Indeed, high-frequency sequences may satisfy
\begin{equation*}
    \norm{u_j}_{H^s}=\varepsilon, \qquad \norm{u_j}_{L^2}\rightarrow0,
\end{equation*}
and therefore
\begin{equation*}
    \geodistance(e,\riemexp u_j)\rightarrow0.
\end{equation*}
Proposition~\ref{compactness of closed normal neighborhood} still gives a convergent subsequence in the full closed normal neighborhood, but its limit may lie away from $\partial \normalhood{e}{\varepsilon}$ and may even be the identity diffeomorphism.

Indeed, it is reasonable to expect that even in the two-dimensional setting the strong hydrodynamical Hopf-Rinow theorem is false. Numerical simulations\footnote{The authors thank Alexander Shnirelman for bringing this work to our attention.} related to the problem of finding global minimizers of the Dirichlet energy suggest a means of establishing this, in combination with Theorem \ref{Hopf-Rinow criteria}, cf. \cite{lan2012the}.

The Grossman-type example below illustrates a further failure which can occur on a general Hilbert manifold: sequences whose excess tends to zero may escape without possessing any convergent subsequence, cf. \cite{grossman1965hilbert}.
\end{remark}

\begin{example}\label{Grossman example}
Let $\{e_j\}_{j\geq0}$ denote the standard orthonormal basis of $\ell^2$, the set of real square-summable sequences, and set
\begin{equation*}
a_0=1, \qquad a_j=1+2^{-j} \quad \text{for }j\geq1.
\end{equation*}
Consider the strong Riemannian Hilbert manifold
\begin{equation*}
\mathcal M=\biggl\{x\in\ell^2:\sum_{j=0}^\infty\frac{x_j^2}{a_j^2}=1\biggr\},
\end{equation*}
equipped with the metric induced by $\ell^2$. It is metrically, and hence geodesically, complete and the image of $\mathbb S^\infty$, the unit sphere in $\ell^2$, under the diffeomorphism
\begin{equation*}
F:\mathbb S^\infty \rightarrow \mathcal M, \qquad F(y)=\sum_{j=0}^\infty a_jy_je_j.
\end{equation*}
The pullback of the standard metric on $\mathcal M$ is
\begin{equation}\label{Grossman pullback metric}
\norm{v}_{\mathcal M}^2=\sum_{j=0}^\infty a_j^2v_j^2
=\norm{v}_{\ell^2}^2+\sum_{j=1}^\infty\bigl(a_j^2-1\bigr)v_j^2.
\end{equation}
Hence, as $\norm{v}_{\mathcal M} \geq \norm{v}_{\ell^2}$, every curve from the north pole $N = e_0$ to the south pole $S= -e_0$ has length at least $\dist^{\mathbb S^\infty}(N,S) = \pi$. On the other hand, for each $j\geq1$ the curve
\begin{equation*}
\gamma_j(t)=\cos(t)e_0+a_j\sin(t)e_j, \qquad 0\leq t\leq\pi,
\end{equation*}
joins $N$ to $S$ in $\mathcal M$ and has length
\begin{equation}\label{Grossman meridian length}
L_j=\int_0^\pi\sqrt{\sin^2(t)+a_j^2\cos^2(t)}\,dt\rightarrow\pi.
\end{equation}
Consequently,
\begin{equation*}
\dist^{\mathcal M}(N,S)=\pi.
\end{equation*}

However, this distance is not attained. Furthermore, there is no nontrivial point metrically between the poles. To see this, take $x=F(y)\in\mathcal M\setminus\{N,S\}$ and set
\begin{equation*}
\theta=\dist^{\mathbb S^\infty}(N,y)=\arccos(y_0)\in(0,\pi).
\end{equation*}
Choose $k\geq1$ for which $y_k\neq0$, and write $\delta_k=a_k^2-1>0$. If $c:[0,1]\to \mathbb S^\infty$ joins $N$ to $y$, reparameterize it to have constant speed with respect to the metric in \eqref{Grossman pullback metric}. The Cauchy-Schwarz inequality gives
\begin{align*}
\mathcal L_g(c)^2 &=\int_0^1\norm{\dot c(t)}_{\ell^2}^2\,dt +\sum_{j=1}^\infty\bigl(a_j^2-1\bigr)\int_0^1\dot c_j(t)^2\,dt\\
&\geq \theta^2+\delta_k y_k^2.
\end{align*}
Taking the infimum over such curves gives
\begin{equation*}
\dist^{\mathcal M}(N,x)\geq\sqrt{\theta^2+\delta_k y_k^2}>\theta.
\end{equation*}
The same argument for curves from $y$ to $S$ gives
\begin{equation*}
\dist^{\mathcal M}(x,S)\geq\sqrt{(\pi-\theta)^2+\delta_k y_k^2}>\pi-\theta.
\end{equation*}
We therefore have
\begin{equation*}
\dist^{\mathcal M}(N,x)+\dist^{\mathcal M}(x,S)>\pi=\dist^{\mathcal M}(N,S),
\end{equation*}
and hence
\begin{equation*}
\bigl\{x\in\mathcal M:\dist^{\mathcal M}(N,x)+\dist^{\mathcal M}(x,S)=\dist^{\mathcal M}(N,S)\bigr\}=\{N,S\}.
\end{equation*}

The failure is caused by escape in the axis index. Indeed, let $x_j=\gamma_j\bigl(\tfrac{\pi}{2}\bigr)=a_je_j$. The two halves of $\gamma_j$ have the same length, and \eqref{Grossman meridian length} gives
\begin{equation*}
0<\dist^{\mathcal M}(N,x_j)+\dist^{\mathcal M}(x_j,S)-\dist^{\mathcal M}(N,S) \leq L_j-\pi\rightarrow0.
\end{equation*}
Hence, these points have excess tending to zero, but for $j\neq k$ we have
\begin{equation*}
\norm{x_j-x_k}_{\ell^2}^2=a_j^2+a_k^2\geq2,
\end{equation*}
so $\{x_j\}$ has no strongly convergent subsequence. Formally adjoining an axis with limiting coefficient $a_\infty=1$ would produce a limiting circular meridian containing points metrically between the poles, as illustrated below. However, that circle is not contained in $\mathcal M$.
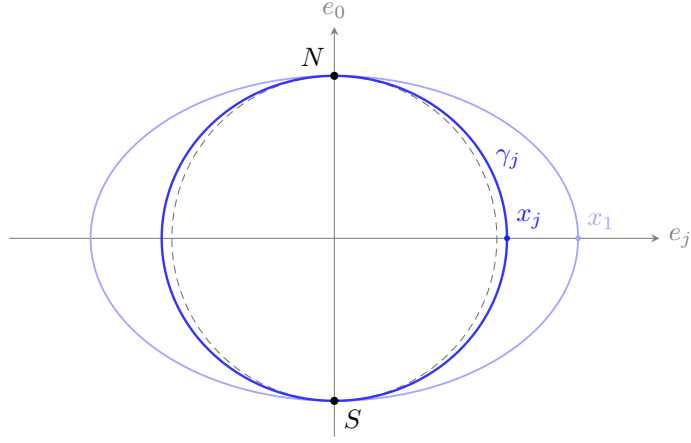
\begin{figure}[H]
\centering
\begin{tikzpicture}[scale=2.15,>=stealth]
\draw[->,gray] (0,-1.22)--(0,1.30) node[above] {$e_0$};
\draw[->,gray] (-2,0)--(2,0) node[right] {$e_j$};
\draw[densely dashed,black!55] plot[domain=0:360,samples=160]
  ({sin(\x)},{cos(\x)});
\draw[blue!35,line width=.7pt] plot[domain=0:360,samples=160]
  ({1.5*sin(\x)},{cos(\x)});
\draw[blue!80,line width=.9pt] plot[domain=0:360,samples=160]
  ({1.0625*sin(\x)},{cos(\x)});
\fill (0,1) circle (.025) node[above left] {$N$};
\fill (0,-1) circle (.025) node[below right] {$S$};
\fill[blue!35] (1.5,0) circle (.018) node[above right] {$x_1$};
\fill[blue!90] (1.0625,0) circle (.018) node[above right] {$x_j$};
\node[blue!90] at (1.07,.49) {$\gamma_j$};
\end{tikzpicture}
\caption{A schematic overlay of Grossman's example. Each solid ellipse is the intersection of $\mathcal M$ with the plane spanned by $\{e_0,e_j\}$. Their planar profiles approach the dashed unit circle, but the corresponding points $\{x_j\}$ escape in mutually orthogonal directions. The dashed circle, which would contain points metrically between the poles, is not contained in $\mathcal M$.}
\label{Grossman figure}
\end{figure}
\end{example}

\enlargethispage{2\baselineskip}

\section{Regularity of the squared geodesic distance in the \texorpdfstring{$L^2$}{L2} topology}\label{regularity of squared geodesic distance}
There are two related ways to interpret the local comparison between the geodesic and ambient $L^2$ distances established in Section~\ref{local metric geometry}.

First, after identifying diffeomorphisms with their canonical displacements, Corollary~\ref{Phi Lipschitz} shows that the geodesic distance and the ambient $L^2$ distance define bi-Lipschitz equivalent metrics on $\normalhood{e}{\varepsilon}$, after shrinking $\varepsilon>0$ if necessary. Equivalently,
\begin{equation*}
\bigl(\normalhood{e}{\varepsilon},\geodistance\bigr)
\hookrightarrow L^2(\T^n,\R^n)
\end{equation*}
is a bi-Lipschitz embedding onto its image. It is natural to ask whether this metric comparison reflects any further ambient $L^2$ structure; for example, whether $\normalhood{e}{\varepsilon}$ can be regarded as a $C^1$ submanifold of $L^2(\T^n,\R^n)$ in the sense of Eells \cite{eells1966a}. However, this cannot be. Indeed, its ambient tangent space at $e$ would have to contain $\overline{T_e\Dmu^s}^{\,L^2}=\Ltwodiv$, while differentiating the volume-preserving identity along ambient $L^2$-differentiable curves gives the reverse inclusion. Hence its tangent space at $e$ would be $\Ltwodiv$. The map
\begin{equation*}
\eta\mapsto\leray_e(\eta-e)
\end{equation*}
would then provide local coordinates whose image is open in $\Ltwodiv$. On the other hand, its image is contained in $T_e\Dmu^s$, since $\eta-e\in H^s$ and $\leray_e$ preserves $H^s$; a contradiction as $T_e\Dmu^s$ has empty interior in $\Ltwodiv$.

A second and more concrete interpretation concerns the ambient regularity of the scalar distance function
\begin{equation*}
f:\normalhood{e}{\varepsilon}\to\R, \qquad
f(\eta)=\geodistance(e,\eta).
\end{equation*}
The triangle inequality and Corollary~\ref{Phi Lipschitz} give
\begin{equation*}
\abs{f(\eta)-f(\varphi)}
\leq\geodistance(\eta,\varphi)
\leq C\norm{\eta-\varphi}_{L^2},
\end{equation*}
while Theorem~\ref{local Shnirelman inequality} gives
\begin{equation*}
\norm{\eta-e}_{L^2}
\leq f(\eta)
\leq C\norm{\eta-e}_{L^2}.
\end{equation*}
Hence $f$ is locally $L^2$-Lipschitz and is quantitatively comparable to the ambient distance from the identity. One should not, however, expect $f$ to be differentiable at $e$. Indeed, the function has the same cusp at the identity as a norm. A natural alternative to consider instead is the squared distance
\begin{equation*}
\sqdistance:\normalhood{e}{\varepsilon}\to\R, \qquad
\sqdistance(\eta)=\geodistance(e,\eta)^2,
\end{equation*}
which removes this obstruction. However, even after squaring the distance, the fact that $\normalhood{e}{\varepsilon}$ has empty interior in $L^2(\T^n,\R^n)$ prevents us from applying the ordinary theory of Fr\'echet differentiable functions directly. A closely related difficulty appears in the Lagrangian treatment of the Euler equation by Bourguignon and Brezis \cite{bourguignon1974remarks}. Their evolution equation is formulated on a closed subset of a Banach space, rather than on an open set, and local solutions are obtained by applying a theorem of Martin \cite{martin1973differential} based on a subtangency condition. In this section we will show, via a Whitney extension type theorem \cite{whitney1934analytic} adapted to the Hilbert space setting \cite{azagra2021extension, johanis2024whitney}, that the squared geodesic distance on $\normalhood{e}{\varepsilon}$ admits a $C^{1,1}$ extension to the full ambient $L^2$ space. In particular, we prove the following.

\begin{theorem}[$C^{1,1}$ extension]\label{Whitney extension of squared distance}
Let $n\geq2$ and $s>\frac n2+1$. There exist $\varepsilon>0$ and a function
\begin{equation*}
    \widetilde{\Phi}\in C^{1,1}\bigl(L^2(\T^n,\R^n)\bigr)
\end{equation*}
such that
\begin{equation*}
    \widetilde{\Phi}\Big\vert_{\normalhood{e}{\varepsilon}}(\eta)=\sqdistance(\eta) =\geodistance(e,\eta)^2.
\end{equation*}
 Moreover, for every $ u = \riemexp^{-1} \eta \in\normalball{\varepsilon}$,
\begin{equation}\label{squared distance differential}
    d\widetilde{\Phi}(\eta)h = 2\ip{d\riemexp(u)u,h}_{L^2}, \qquad h\in L^2(\T^n,\R^n).
\end{equation}
\end{theorem}

The extension $\widetilde{\Phi}$ is not canonical and need not be nonnegative or retain any interpretation as a squared distance away from $\normalhood{e}{\varepsilon}$. In particular, it does not extend either the Riemannian exponential map or the hydrodynamical geometry to arbitrary $L^2$ maps. The $H^s$ differential determines the restriction of the ambient differential to the $L^2$-completion of the tangent space. We choose its extension to the ambient space to vanish on the orthogonal complement of that completed tangent space.

By Theorem~\ref{local geodesic convexity}, if $\eta=\riemexp u$ with $u\in\normalball{\varepsilon}$, then
\begin{equation*}
    \sqdistance(\eta)=\norm{\riemexp^{-1}\eta}_{L^2}^2=\norm{u}_{L^2}^2.
\end{equation*}
Differentiating this identity in the $H^s$ manifold topology and applying the Gauss Lemma gives
\begin{equation*}
    d\sqdistance(\eta)V = 2\ip{u,d(\riemexp^{-1})(\eta)V}_{L^2} = 2\ip{d\riemexp(u)u,V}_{L^2}, \qquad V\in T_\eta\Dmu^s.
\end{equation*}
If $\geod(t)=\riemexp(tu)$, then $d\riemexp(u)u=\dot{\geod}(1)$ and, by constant speed of geodesics we have
\begin{equation*}
    \abs{d\sqdistance(\eta)V} \leq 2\norm{u}_{L^2}\norm{V}_{L^2}, \qquad V\in T_\eta\Dmu^s.
\end{equation*}
By density in $\overline{T_\eta\Dmu^s}^{\,L^2}$ and the $L^2$-orthogonal decomposition at $\eta$, extending this differential by zero on the normal complement gives the canonical candidate
\begin{equation}\label{ambient differential}
    d_{L^2}\sqdistance(\eta)h = 2 \ip{d\riemexp(u)u, h}_{L^2}, \quad h \in L^2(\T^n, \R^n).
\end{equation}
We further establish a uniform $L^2$ lower bound for the first derivative of the Riemannian exponential map.

\begin{proposition}\label{surjectivity modulus}
There exists $\varepsilon>0$ and $\rho(\varepsilon)>0$ such that
\begin{equation}\label{surjectivity modulus estimate}
\norm{d\riemexp(u_0)w}_{L^2} \geq \rho(\varepsilon)\norm{w}_{L^2}
\end{equation}
for all $u_0\in \normalball{\varepsilon}$ and all $w\in T_e\Dmu^s(\T^n)$.
\end{proposition}

Before proceeding to the proof we recall the Lie geometric decomposition of the first derivative of the $L^2$ exponential map, cf. \cite{misiolek2010fredholm}. Let $u_0 \in \normalball{\varepsilon}$ with $\geod(t) = \riemexp tu_0$ and consider the operators
\begin{equation*}
    \Lambda(t)w = \Adstar_{\gamma(t)} \Ad_{\gamma(t)}w, \quad \text{and} \quad K_{u_0}w = \adstar_wu_0 \quad \text{for} \ w \in T_e\Dmu^s.
\end{equation*}
We will regard $\Lambda(t)$ and $K_{u_0}$ as operators on $\Ltwodiv$. As each $\gamma(t)$ is volume-preserving we have
\begin{equation*}
    \Lambda(t)=\leray_e D\gamma(t)^T D\gamma(t), \qquad \ip{\Lambda(t)w,w}_{L^2}=\norm{D\gamma(t)w}_{L^2}^2 \geq \frac{\norm{w}_{L^2}^2}{\norm{D\gamma(t)^{-1}}_{L^\infty}^2}.
\end{equation*}
Hence $\Lambda(t)$ is bounded, self-adjoint, coercive and has a bounded inverse on $\Ltwodiv$. Lemma~\ref{algebra coadjoint estimate} gives that
\begin{equation}\label{K_{u_0} estimate}
    \norm{K_{u_0}w}_{L^2} \leq C\norm{u_0}_{H^s}\norm{w}_{L^2} \quad \text{for} \ u_0, w \in T_e\Dmu^s
\end{equation}
and consequently $K_{u_0}$ has a bounded extension to $\Ltwodiv$. The function
\begin{equation}\label{left-translated Jacobi}
    y : [0,1] \times T_e\Dmu^s \rightarrow \Ltwodiv \ ; \ (t, w_0) \mapsto D\gamma(t)^{-1}d\riemexp(tu_0)tw_0
\end{equation}
satisfies the following integral equation in $\Ltwodiv$
\begin{equation}\label{first derivative decomp}
    y(t, w_0) = \int_0^t \Lambda(\tau)^{-1}w_0 \, d\tau - \int_0^t\Lambda(\tau)^{-1} K_{u_0} y(\tau, w_0)\, d\tau.
\end{equation}
\begin{proof}[Proof of Proposition \ref{surjectivity modulus}]
    For sufficiently small $\varepsilon>0$, the curve $t \mapsto \riemexp tu_0$ remains uniformly close to the identity in $C^1$. Hence, from \eqref{left-translated Jacobi}, we need only prove an analogous estimate for $y$. Shrinking $\varepsilon>0$ if necessary we have that
    \begin{equation}\label{uniform operator norm bounds}
        a(\varepsilon) = \sup_{\substack{u_0\in\normalball{\varepsilon}\\0\leq t\leq1}} \opnorm{\Lambda_{u_0}(t)^{-1}}{\Ltwodiv}, \qquad b(\varepsilon) = \sup_{\substack{u_0\in\normalball{\varepsilon}\\0\leq t\leq1}} \opnorm{\Ad_{\gamma_{u_0}(t)}}{\Ltwodiv},
    \end{equation}
    are finite uniform constants depending only on $\varepsilon$, with $n$ and $s$ fixed. Here $\gamma_{u_0}(t)=\riemexp(tu_0)$ and $\Lambda_{u_0}(t)$ denotes the corresponding operator. The bounds follow from uniform $C^1$ closeness to the identity; in particular both constants remain bounded as $\varepsilon\to0$. Combining this with the decomposition \eqref{first derivative decomp} and \eqref{K_{u_0} estimate} we have
    \begin{equation*}
        \norm{y(t,w_0)}_{L^2} \leq a(\varepsilon) t \norm{w_0}_{L^2} + C a(\varepsilon) \norm{u_0}_{H^s} \int_0^t \norm{y(\tau, w_0)}_{L^2} \, d\tau.
    \end{equation*}
    By Gr\"onwall we have
    \begin{equation}\label{y upper bound}
        \norm{y(t, w_0)}_{L^2} \leq a(\varepsilon)t e^{C a(\varepsilon) \norm{u_0}_{H^s} t}\norm{w_0}_{L^2}.
    \end{equation}
    Hence, as $\ip{\Lambda(t)^{-1}w_0, w_0}_{L^2} = \norm{\bigl(\Adstar_{\gamma(t)}\bigr)^{-1}w_0}_{L^2}^2 \geq \frac{1}{b(\varepsilon)^2}\norm{w_0}_{L^2}^2$, we estimate
    \begin{align*}
        \norm{y(1, w_0)}_{L^2} &\geq \norm{\int_0^1 \Lambda(t)^{-1}w_0 \, dt}_{L^2} - \int_0^1 \norm{\Lambda(t)^{-1} K_{u_0} y(t, w_0)}_{L^2} \, dt \\
        &\geq \biggl( \frac{1}{b(\varepsilon)^2} - Ca(\varepsilon)^2\norm{u_0}_{H^s}e^{Ca(\varepsilon)\norm{u_0}_{H^s}} \biggr) \norm{w_0}_{L^2}.
    \end{align*}
    Taking $\varepsilon>0$ sufficiently small yields the result.
\end{proof}

The estimate \eqref{y upper bound} and the uniform $L^\infty$ bound on $D\gamma(1)$ further give
\begin{equation*}
    \norm{d\riemexp(u_0)(w)}_{L^2}\leq C\norm{w}_{L^2}, \quad w \in T_e\Dmu^s.
\end{equation*}
Hence, by density, $d\riemexp(u_0)$ extends uniquely to a bounded operator from $\Ltwodiv$ to $L^2(\T^n,\R^n)$, uniformly for $u_0\in\normalball{\varepsilon}$. The next required ingredient provides control on the variation in operator norm of this extension.

\begin{proposition}\label{operator continuity of differential}
    Let $\frac{n}{2} + 1 < s' < s$. There exists $\varepsilon>0$ and $C>0$ such that
    \begin{equation}\label{operator continuity estimate}
        \opnorm{d\riemexp(u) - d\riemexp(v)}{\Ltwodiv \rightarrow L^2} \leq C \norm{u-v}_{H^{s'}}
    \end{equation}
    for all $u, v \in \normalball{\varepsilon}$ where the target tangent spaces are identified by the canonical trivialization of the tangent bundle of $\T^n$. 
\end{proposition}

\begin{proof}
    Choose $\varepsilon>0$ sufficiently small so that the $L^2$ exponential map is a smooth diffeomorphism onto its image on $\normalball{\varepsilon}$ and $\normalball[s']{\varepsilon}$ in the $H^s$ and $H^{s'}$ topologies, respectively. Shrinking $\varepsilon$ further if necessary, continuity at the origin ensures that its differential is uniformly bounded in the $H^{s'}$ operator norm on $\normalball[s']{\varepsilon}$. For $u \in \normalball{\varepsilon}$ and $w_0 \in T_e\Dmu^s$ we define $\gamma_u(t) = \riemexp tu$, $\Lambda_u(t) = \Adstar_{\gamma_u(t)}\Ad_{\gamma_u(t)}$ and $y_u(t, w_0) = D\gamma_u(t)^{-1} d\riemexp(tu)(tw_0)$. We may then decompose
    \begin{align*}
        \norm{d\riemexp(u)w_0 - d\riemexp(v)w_0}_{L^2} &= \norm{D\gamma_u(1) y_u(1, w_0) - D\gamma_v(1)y_v(1, w_0)}_{L^2} \\
        &\leq \norm{\big(D\gamma_u(1) - D\gamma_v(1)\big)y_u(1,w_0)}_{L^2} + \norm{D\gamma_v(1)\big(y_u(1,w_0)-y_v(1,w_0)\big)}_{L^2} \\
        &\leq \norm{D\gamma_u(1) - D\gamma_v(1)}_{L^\infty}\norm{y_u(1,w_0)}_{L^2} + \norm{D\gamma_v(1)}_{L^\infty}\norm{y_u(1,w_0)-y_v(1,w_0)}_{L^2}.
    \end{align*}
    By the smoothness of the Riemannian exponential map and the Sobolev lemma we have that
    \begin{equation*}
        \sup_{0\leq t\leq1}\norm{D\gamma_u(t)-D\gamma_v(t)}_{L^\infty}
        \leq C\norm{u-v}_{H^{s'}},
    \end{equation*}
    and, combining this with \eqref{y upper bound}, there exists a positive constant $C>0$ such that
    \begin{equation}\label{first derivative difference}
        \norm{d\riemexp(u)w_0 - d\riemexp(v)w_0}_{L^2} \leq C \norm{u-v}_{H^{s'}}\norm{w_0}_{L^2} + C\norm{y_u(1, w_0) - y_v(1,w_0)}_{L^2}.
    \end{equation}
    To estimate the final term we consider the differences of the decompositions given in \eqref{first derivative decomp}. For every $t\in[0,1]$ this gives
    \begin{equation}\label{y difference}
        y_u(t, w_0) - y_v(t,w_0) = \int_0^t \big(\Lambda_u(\tau)^{-1} - \Lambda_v(\tau)^{-1} \big) w_0 \, d\tau - \int_0^t \Bigl(\Lambda_u(\tau)^{-1} K_u y_u(\tau, w_0) - \Lambda_v(\tau)^{-1} K_v y_v(\tau, w_0)\Bigr) \, d\tau.
    \end{equation}
    Suppressing $\tau$ and $w_0$, we decompose the second integrand as
    \begin{equation*}
        \Lambda_u^{-1} K_u y_u - \Lambda_v^{-1} K_v y_v = \big( \Lambda_u^{-1} - \Lambda_v^{-1} \big)K_u y_u + \Lambda_v^{-1}\big( K_u - K_v\big)y_u + \Lambda_v^{-1} K_v \big( y_u -y_v \big).
    \end{equation*}
    Writing $\Lambda_u = \leray_e D\gamma_u^T D\gamma_u$ we estimate, again by smoothness of $L^2$ exponential map and the Sobolev lemma
    \begin{equation*}
        \opnorm{\Lambda_u(t) - \Lambda_v(t)}{\Ltwodiv} \leq \opnorm{D\gamma_u(t)^T D\gamma_u(t) - D\gamma_v(t)^T D\gamma_v(t)}{L^2} \leq C \norm{u-v}_{H^{s'}}
    \end{equation*}
    which, using \eqref{uniform operator norm bounds}, allows us to bound
    \begin{equation*}
        \opnorm{\Lambda_u(t)^{-1} - \Lambda_v(t)^{-1}}{\Ltwodiv} =\opnorm{\Lambda_u(t)^{-1} \big( \Lambda_v(t) - \Lambda_u(t)\big)\Lambda_v(t)^{-1}}{\Ltwodiv} \leq C\norm{u-v}_{H^{s'}}.
    \end{equation*}
    Further note that the proof of \eqref{K_{u_0} estimate}, applied at regularity $s'>\frac n2+1$ to $u-v$, gives
    \begin{equation*}
        \opnorm{K_u - K_v}{\Ltwodiv} \leq C\norm{u-v}_{H^{s'}}.
    \end{equation*}
    Combining all of the above, uniformly for $t\in[0,1]$, we obtain
    \begin{equation*}
        \norm{y_u(t, w_0) - y_v(t, w_0)}_{L^2} \leq C\norm{u-v}_{H^{s'}} \norm{w_0}_{L^2} + C \int_0^t \norm{y_u(\tau, w_0) - y_v(\tau, w_0)}_{L^2} \, d\tau.
    \end{equation*}
    Applying Gr{\"o}nwall's inequality then yields
    \begin{equation*}
        \norm{y_u(1, w_0) - y_v(1, w_0)}_{L^2} \leq C \norm{u - v}_{H^{s'}} \norm{w_0}_{L^2}.
    \end{equation*}
    The proposition then follows from \eqref{first derivative difference} and the density of $T_e\Dmu^s$ in $\Ltwodiv$.
\end{proof}

\begin{corollary}\label{mixed second derivative estimate}
Let $\frac n2+1<s'<s$. There exist $\varepsilon>0$ and $C>0$ such that
\begin{equation}\label{mixed second derivative bound}
    \norm{d^2\riemexp(u)(v,w)}_{L^2} \leq C\norm{v}_{H^{s'}}\norm{w}_{L^2}
\end{equation}
for all $u\in\normalball{\varepsilon}$ and $v,w\in T_e\Dmu^s(\T^n)$. Here $d^2\riemexp$ denotes the ordinary second derivative in canonical lift coordinates, rather than the covariant Hessian.
\end{corollary}

\enlargethispage{4\baselineskip}

\begin{proof}
Choose $\varepsilon>0$ and $C>0$ as in Proposition~\ref{operator continuity of differential}. Fix $u\in\normalball{\varepsilon}$ and $v,w\in T_e\Dmu^s$. For all sufficiently small $\abs r$ we have $u+rv\in\normalball{\varepsilon}$, and smoothness in canonical lift coordinates gives
\begin{equation*}
    \frac{\bigl(d\riemexp(u+rv)-d\riemexp(u)\bigr)(w)}{r} \rightarrow d^2\riemexp(u)(v,w) \quad\text{in }H^s,
\end{equation*}
and hence in $L^2$, as $r\to0$. Moreover, Proposition~\ref{operator continuity of differential} gives, for all sufficiently small nonzero $r$,
\begin{equation*}
    \norm{\frac{\bigl(d\riemexp(u+rv)-d\riemexp(u)\bigr)(w)}{r}}_{L^2} \leq C\frac{\norm{rv}_{H^{s'}}}{\abs r}\norm{w}_{L^2} =C\norm{v}_{H^{s'}}\norm{w}_{L^2}.
\end{equation*}
Passing to the limit proves \eqref{mixed second derivative bound}.
\end{proof}

We now record an $L^2$ Taylor expansion for the Riemannian exponential map. In order to consider differences of diffeomorphisms within $\normalhood{e}{\varepsilon}$ we identify each element with its previously described canonical displacement. We will however, for the sake of exposition, suppress this notation in this section. When $d\riemexp(u)$ is regarded as an $H^s$- or $L^2$-valued operator, it denotes the intrinsic differential represented as an $\R^n$-valued field under the canonical trivialization of the tangent bundle of $\T^n$.

\begin{corollary}\label{L^2 Taylor expansion}
    Let $\frac{n}{2}+1 < s' < s$. There exists $\varepsilon>0$ and $C>0$ such that for all $u_0, u_0 + h \in \normalball{\varepsilon}$ we have
    \begin{equation}\label{L^2 Taylor estimate}
        \norm{\riemexp(u_0 + h) - \riemexp(u_0) - d\riemexp(u_0)h}_{L^2} \leq C \norm{h}_{H^{s'}}\norm{h}_{L^2}.
    \end{equation}
    In particular, if $\norm{h}_{L^2} \to 0$ with $u_0 + h \in \normalball{\varepsilon}$, then
    \begin{equation*}
        \norm{\riemexp(u_0 + h) - \riemexp(u_0) - d\riemexp(u_0)h}_{L^2} = o\big(\norm{h}_{L^2}\big).
    \end{equation*}
\end{corollary}

\begin{proof}
    As $\normalball{\varepsilon}$ is convex we have
    \begin{equation*}
        \riemexp(u_0 + h) - \riemexp(u_0) - d\riemexp(u_0)h = \int_0^1 \big(d\riemexp(u_0 + \theta h) - d\riemexp(u_0)\big)h \, d\theta.
    \end{equation*}
    The inequality \eqref{L^2 Taylor estimate} then follows by Proposition \ref{operator continuity of differential}. Since $u_0$ and $u_0+h$ lie in the fixed $H^s$ ball, $h$ remains bounded in $H^s$. The asserted little-$o$ statement then follows from Sobolev interpolation
    \begin{equation*}
        \norm{h}_{H^{s'}} \leq C \norm{h}_{H^s}^{\frac{s'}{s}} \norm{h}_{L^2}^{1 - \frac{s'}{s}}.
    \end{equation*}
\end{proof}

Next we establish an $L^2$-Lipschitz estimate for the inverse of the Riemannian exponential map.

\begin{lemma}\label{L^2 continuity of inverse exponential map}
There exists $\varepsilon>0$ and $C>0$ such that
\begin{equation}\label{L^2 inverse exponential estimate}
    \norm{u-v}_{L^2}\leq C\norm{\riemexp u-\riemexp v}_{L^2}
\end{equation}
for all $u,v\in\normalball{\varepsilon}$. In particular,
\begin{equation*}
    \riemexp^{-1}:\normalhood{e}{\varepsilon}\to\normalball{\varepsilon}
\end{equation*}
is Lipschitz with respect to the ambient $L^2$ topologies.
\end{lemma}

\begin{proof}
Fix $\varepsilon_0>0$ sufficiently small that Proposition~\ref{surjectivity modulus} and Corollary~\ref{L^2 Taylor expansion} hold on $\normalball{\varepsilon_0}$. Let $\rho_0>0$ denote the constant in \eqref{surjectivity modulus estimate} on this ball. For $u,v\in\normalball{\varepsilon}$, where $\varepsilon\leq\varepsilon_0$, set $h=v-u$. By Corollary~\ref{L^2 Taylor expansion},
\begin{equation*}
    \riemexp v-\riemexp u=d\riemexp(u)h+R(u,h),
\end{equation*}
where
\begin{equation*}
    \norm{R(u,h)}_{L^2}\leq C\norm{h}_{H^{s'}}\norm{h}_{L^2}.
\end{equation*}
Since $u,v\in\normalball{\varepsilon}$, the continuous embedding $H^s\hookrightarrow H^{s'}$ gives
\begin{equation*}
    \norm{h}_{H^{s'}}\leq C\norm{h}_{H^s}\leq2C\varepsilon.
\end{equation*}
Consequently, Proposition~\ref{surjectivity modulus} gives
\begin{equation*}
    \norm{\riemexp v-\riemexp u}_{L^2}\geq\bigl(\rho_0-C\varepsilon\bigr)\norm{h}_{L^2}.
\end{equation*}
Shrinking $\varepsilon>0$ so that the error is at most $\rho_0/2$, we obtain
\begin{equation*}
    \norm{\riemexp v-\riemexp u}_{L^2}\geq\frac{\rho_0}{2}\norm{u-v}_{L^2},
\end{equation*}
which yields the result.
\end{proof}

We are now ready to prove the main result of this section.

\begin{proof}[Proof of Theorem~\ref{Whitney extension of squared distance}]
Fix $\frac n2+1<s'<s$ and choose $\varepsilon>0$ sufficiently small that the preceding estimates hold. Applying Proposition~\ref{operator continuity of differential} with one basepoint equal to $0$ gives a uniform bound
\begin{equation*}
    \norm{d\riemexp(u)w}_{L^2}\leq C\norm{w}_{L^2},
    \qquad u\in\normalball{\varepsilon},\quad w\in T_e\Dmu^s.
\end{equation*}
Let $\eta=\riemexp u,$ $\xi=\riemexp v$ and $h=v-u$ where $u,v\in\normalball{\varepsilon}$. Since the ball is convex, $u + \theta h \in\normalball{\varepsilon}$ for all $0 \leq \theta \leq 1$. Differentiating in canonical lift coordinates gives
\begin{equation*}
    \frac{d}{d\theta}\bigl(d\riemexp(u + \theta h)(u + \theta h)\bigr)
    =d^2\riemexp(u + \theta h)(h,u + \theta h)+d\riemexp(u + \theta h)(h).
\end{equation*}
By symmetry of $d^2\riemexp$ and Corollary~\ref{mixed second derivative estimate}, we have
\begin{equation*}
    \norm{d^2\riemexp(u + \theta h)(h,u + \theta h)}_{L^2} \leq C\norm{u + \theta h}_{H^{s'}}\norm{h}_{L^2} \leq C\norm{h}_{L^2}
\end{equation*}
and hence, integrating the preceding derivative and using \eqref{ambient differential}, we obtain
\begin{equation*}
    \opnorm{d_{L^2}\sqdistance(\riemexp(u + \theta h))-d_{L^2}\sqdistance(\eta)}{L^2} = 2\norm{d\riemexp(u + \theta h)(u + \theta h)-d\riemexp(u)u}_{L^2} \leq C\theta\norm{h}_{L^2}.
\end{equation*}
Taking $\theta=1$ and applying \eqref{L^2 inverse exponential estimate} gives
\begin{equation}\label{Whitney derivative estimate}
    \opnorm{d_{L^2}\sqdistance(\xi)-d_{L^2}\sqdistance(\eta)}{L^2}
    \leq C\norm{\xi-\eta}_{L^2}.
\end{equation}

Furthermore, we have
\begin{equation*}
    \sqdistance(\xi)-\sqdistance(\eta)-d_{L^2}\sqdistance(\eta)(\xi-\eta) = \int_0^1
    \bigl(d_{L^2}\sqdistance(\riemexp(u + \theta h))-d_{L^2}\sqdistance(\eta)\bigr)
    \bigl(d\riemexp(u + \theta h)(h)\bigr)\,d\theta.
\end{equation*}
The preceding bounds then give
\begin{equation*}
    \abs{\sqdistance(\xi)-\sqdistance(\eta)-d_{L^2}\sqdistance(\eta)(\xi-\eta)}
    \leq C\int_0^1\theta\norm{h}_{L^2}^2\,d\theta
    \leq C\norm{h}_{L^2}^2.
\end{equation*}
Another application of \eqref{L^2 inverse exponential estimate} yields
\begin{equation}\label{Whitney function estimate}
    \abs{\sqdistance(\xi)-\sqdistance(\eta)-d_{L^2}\sqdistance(\eta)(\xi-\eta)}
    \leq C\norm{\xi-\eta}_{L^2}^{2}.
\end{equation}
The estimates \eqref{Whitney derivative estimate} and \eqref{Whitney function estimate} are the Whitney-Glaeser conditions with modulus $\omega(t)=t$. The Hilbert-space Whitney extension theorem \cite[Theorem~2]{johanis2024whitney}, which applies to arbitrary subsets of the ambient space, then gives a function
\begin{equation*}
    \widetilde{\Phi}\in C^{1,1}\bigl(L^2(\T^n,\R^n)\bigr)
\end{equation*}
whose value and differential on $\normalhood{e}{\varepsilon}$ are $\sqdistance$ and $d_{L^2}\sqdistance$, respectively. This establishes \eqref{squared distance differential} and completes the proof.
\end{proof}

The next result demonstrates that there is a limit to the regularity of possible extensions preserving the values and differential of the squared geodesic distance.

\begin{proposition}\label{no C4 extension with prescribed differential}
There is no $\varepsilon>0$ for which the squared geodesic distance
\begin{equation*}
    \sqdistance:\normalhood{e}{\varepsilon}\rightarrow\R, \qquad \sqdistance(\eta)=\geodistance(e,\eta)^2,
\end{equation*}
admits a $C^4$ extension $\widetilde{\Phi}$ to an $L^2$-open set containing $\normalhood{e}{\varepsilon}$ such that
\begin{equation}\label{C4 prescribed differential}
    d\widetilde{\Phi}(\eta)(h) = d_{L^2}\sqdistance(\eta)(h) = 2\ip{d\riemexp(u)(u),h}_{L^2}, \qquad u=\riemexp^{-1}\eta,
\end{equation}
for every $\eta\in\normalhood{e}{\varepsilon}$ and every $h\in L^2(\T^n,\R^n)$. Here the normal neighborhood is identified with its canonical displacements in $L^2(\T^n,\R^n)$, and all derivatives are ordinary ambient Fr\'echet derivatives.
\end{proposition}

\begin{proof}
Suppose that such an extension exists. Fix $u_0\in\normalball{\varepsilon}$ and let $\gamma(t)=\riemexp(tu_0)$ be the corresponding Euler geodesic, represented in canonical lift coordinates. For every fixed $h\in L^2(\T^n,\R^n)$ and all $t$ sufficiently close to zero, \eqref{C4 prescribed differential} gives
\begin{equation}\label{C4 radial differential}
    d\widetilde{\Phi}(\gamma(t))(h) = 2\ip{d\riemexp(tu_0)(tu_0),h}_{L^2} = 2t\ip{\dot\gamma(t),h}_{L^2}.
\end{equation}
In particular, $d\widetilde{\Phi}(e)(h)=0$. Differentiating \eqref{C4 radial differential} once gives
\begin{equation*}
    d^2\widetilde{\Phi}(\gamma(t))(\dot\gamma(t),h) = 2\ip{\dot\gamma(t),h}_{L^2} + 2t\ip{\ddot\gamma(t),h}_{L^2}.
\end{equation*}
Evaluating at $t=0$ yields
\begin{equation}\label{C4 second derivative identity}
    d^2\widetilde{\Phi}(e)(u_0,h)=2\ip{u_0,h}_{L^2}.
\end{equation}
Differentiating twice more and using symmetry of the ambient derivatives gives
\begin{equation*}
    d^4\widetilde{\Phi}(\gamma(t)) (\dot\gamma(t),\dot\gamma(t),\dot\gamma(t),h) + 3d^3\widetilde{\Phi}(\gamma(t)) (\ddot\gamma(t),\dot\gamma(t),h) + d^2\widetilde{\Phi}(\gamma(t))(\gamma^{(3)}(t),h) = 6\ip{\gamma^{(3)}(t),h}_{L^2} + 2t\ip{\gamma^{(4)}(t),h}_{L^2}.
\end{equation*}
Setting
\begin{equation*}
    v_0=\ddot\gamma(0)=-\nabla p(0), \qquad w_0=\gamma^{(3)}(0),
\end{equation*}
where $p$ is the pressure associated with $u_0$, at $t=0$, the preceding identity becomes
\begin{equation}\label{C4 fourth derivative identity}
    d^4\widetilde{\Phi}(e)(u_0,u_0,u_0,h) + 3d^3\widetilde{\Phi}(e)(v_0,u_0,h) + d^2\widetilde{\Phi}(e)(w_0,h) = 6\ip{w_0,h}_{L^2}.
\end{equation}

Since $\gamma$ has constant $L^2$ speed,
\begin{equation*}
    0=\frac12\frac{d^2}{dt^2}\norm{\dot\gamma(t)}_{L^2}^2 =\ip{\gamma^{(3)}(t),\dot\gamma(t)}_{L^2} + \norm{\ddot\gamma(t)}_{L^2}^2.
\end{equation*}
Consequently,
\begin{equation}\label{C4 constant speed identity}
    \ip{w_0,u_0}_{L^2}=-\norm{v_0}_{L^2}^2.
\end{equation}
By symmetry and \eqref{C4 second derivative identity}, with $h=w_0$,
\begin{equation*}
    d^2\widetilde{\Phi}(e)(w_0,u_0) =d^2\widetilde{\Phi}(e)(u_0,w_0) = 2\ip{u_0,w_0}_{L^2} = -2\norm{v_0}_{L^2}^2.
\end{equation*}
Taking $h=u_0$ in \eqref{C4 fourth derivative identity} therefore gives
\begin{equation*}
    d^4\widetilde{\Phi}(e)(u_0,u_0,u_0,u_0) + 3d^3\widetilde{\Phi}(e)(v_0,u_0,u_0) = -4\norm{v_0}_{L^2}^2.
\end{equation*}
Boundedness of the third and fourth ambient derivatives at $e$ now gives constants $C_3,C_4>0$, independent of $u_0$, such that
\begin{equation}\label{C4 obstruction estimate}
    4\norm{v_0}_{L^2}^2 \leq C_4\norm{u_0}_{L^2}^4 + 3C_3\norm{v_0}_{L^2}\norm{u_0}_{L^2}^2.
\end{equation}
Applying this to the initial conditions $\{u_k\} \subset \normalball{\varepsilon}$ constructed in the proof of Proposition \ref{local ratio unbounded} gives
\begin{equation*}
    8 \pi k \leq \frac{C_4}{2 \pi k} + 3C_3
\end{equation*}
for all $k \geq 1$; a contradiction.
\end{proof}

The preceding proposition rules out an ambient $C^4$ extension of the squared geodesic distance which preserves both its values and its prescribed ambient differential. Nevertheless, we expect analogous mixed Sobolev estimates to hold for higher derivatives of $\riemexp$. Lie theoretic decompositions of these higher-order derivatives provide a possible route to establishing such estimates, cf. \cite{heslin2026exact}.

Finally, we apply the estimates developed in this section to show that the intrinsic and ambient distances agree to first order at every point of a sufficiently small normal neighborhood.

\begin{corollary}
    There exists $\varepsilon>0$ such that for every fixed $\eta \in \normalhood{e}{\varepsilon}$ and every sequence $\{\eta_j\} \subset \normalhood{e}{\varepsilon}\setminus \{\eta\}$ with $\norm{\eta - \eta_j}_{L^2} \to 0$ we have
    \begin{equation}\label{local asymptotic ratio}
        \frac{\geodistance(\eta, \eta_j)}{\norm{\eta - \eta_j}_{L^2}} \to 1.
    \end{equation}
\end{corollary}

\begin{proof}
    First assume that $\eta = e$ and $\eta_j = \riemexp u_j$. By Lemma~\ref{L^2 continuity of inverse exponential map} we have $\norm{u_j}_{L^2} \to 0$. Hence, by Corollary~\ref{L^2 Taylor expansion}, as $d\riemexp(0) = \mathrm{Id}$, we have
    \begin{equation*}
        \norm{\eta_j - e - u_j}_{L^2} = o\big(\norm{u_j}_{L^2}\big).
    \end{equation*}
    Consequently
    \begin{equation*}
        \norm{\eta_j - e}_{L^2} = \big( 1 + o(1) \big)\norm{u_j}_{L^2}
    \end{equation*}
    and the result follows for $\eta = e$.

    \newpage

    Finally, shrink $\varepsilon>0$ so that $\eta,\xi\in\normalhood{e}{\varepsilon}$ implies $\xi\circ\eta^{-1}$ lies in the normal neighborhood on which the preceding argument at the identity applies. For arbitrary $\eta\in\normalhood{e}{\varepsilon}$ and any sequence $\{\eta_j\}\subset\normalhood{e}{\varepsilon}\setminus\{\eta\}$ with $\norm{\eta-\eta_j}_{L^2}\to0$, the translated sequence $\xi_j=\eta_j\circ\eta^{-1}$ lies in that neighborhood. By right invariance and the volume-preserving property, we have
    \begin{equation*}
        \lim_{j \to \infty} \frac{\geodistance(\eta, \eta_j)}{\norm{\eta - \eta_j}_{L^2}} = \lim_{j \to \infty} \frac{\geodistance(e, \xi_j)}{\norm{\xi_j - e}_{L^2}} = 1.
    \end{equation*}
\end{proof}

\subsection{Second variations of the exponential map}
The proof of Theorem~\ref{Whitney extension of squared distance} uses the operator estimate for the first derivative and the symmetry of the ordinary second derivative in canonical lift coordinates. We now record the equation obtained by differentiating the Jacobi equation once more. This equation is not needed in the proof above, but it describes the covariant second derivative of the exponential map and suggests an alternative route to the operator estimate in Proposition~\ref{operator continuity of differential}. Since the $L^2$ metric is weak, the calculation below is understood along variations of sufficient regularity that all the displayed covariant derivatives of the curvature tensor are defined.

Let $u_0 \in \normalball{\varepsilon}$ with $v_0,w_0\in T_e\Dmu^s$ and consider the two-parameter family of geodesics
\begin{equation}\label{two parameter geodesic variation}
    \Gamma(t,r,s)=\riemexp\big(t(u_0+rv_0+sw_0)\big).
\end{equation}
Here $0 \leq t \leq 1$ and $r, s$ are sufficiently small so that $u_0 + rv_0 + sw_0 \in \normalball{\varepsilon}$. We write
\begin{equation*}
    T=\partial_t\Gamma,\qquad V=\partial_r\Gamma,\qquad W=\partial_s\Gamma
\end{equation*}
for the corresponding vector fields along the variation, and set
\begin{equation*}
    \geod(t)=\Gamma(t,0,0)=\riemexp tu_0.
\end{equation*}
Here, for example, $\partial_r\Gamma=d\Gamma(\partial_r)$ is an intrinsic tangent vector to $\Dmu^s$. No lift of the exponential map is required. Along $\geod$ the two first variation fields are
\begin{equation}\label{first variation fields second variation}
    V(t)=d\riemexp(tu_0)(tv_0),\qquad W(t)=d\riemexp(tu_0)(tw_0).
\end{equation}
Equivalently, in the notation of \eqref{left-translated Jacobi},
\begin{equation*}
    V(t)=D\geod(t)y(t,v_0),\qquad W(t)=D\geod(t)y(t,w_0).
\end{equation*}
Both fields satisfy the Jacobi equation
\begin{equation}\label{Jacobi equations first variations}
    \frac{D^2V}{dt^2}+R(V,\dot\geod)\dot\geod=0, \qquad \frac{D^2W}{dt^2}+R(W,\dot\geod)\dot\geod=0,
\end{equation}
with initial conditions
\begin{equation}\label{first variation initial conditions}
    V(0)=W(0)=0, \qquad \frac{DV}{dt}(0)=v_0, \qquad \frac{DW}{dt}(0)=w_0.
\end{equation}

We define the covariant second variation along $\geod$ by
\begin{equation}\label{covariant second variation}
    Z(t)=\biggl.\frac{D}{ds}\biggr|_{s=0}V(t,0,s) =\bigl.\nabla_WV\bigr|_{(t,0,0)}.
\end{equation}
Since the domain of $\riemexp$ is equipped with its flat connection, we have
\begin{equation}\label{covariant Hessian exponential}
    Z(t)=(\nabla d\riemexp)_{tu_0}(tw_0,tv_0).
\end{equation}
In particular,
\begin{equation*}
    Z(1)=(\nabla d\riemexp)_{u_0}(w_0,v_0).
\end{equation*}
This is the intrinsic meaning of the second derivative of the manifold-valued map $\riemexp$. We use the curvature convention
\begin{equation}\label{curvature convention}
    R(X,Y)Z=\nabla_X\nabla_YZ-\nabla_Y\nabla_XZ-\nabla_{[X,Y]}Z.
\end{equation}

\begin{proposition}\label{second variation equation}
The covariant second variation $Z$ satisfies
\begin{equation}\label{inhomogeneous Jacobi equation}
    \frac{D^2Z}{dt^2}+R(Z,\dot\geod)\dot\geod+2R(W,\dot\geod)\frac{DV}{dt} +2R(V,\dot\geod)\frac{DW}{dt} +(\nabla_{\dot\geod}R)(W,\dot\geod)V +(\nabla_WR)(V,\dot\geod)\dot\geod=0,
\end{equation}
with initial conditions
\begin{equation}\label{second variation initial conditions}
    Z(0)=0,\qquad \frac{DZ}{dt}(0)=0.
\end{equation}
\end{proposition}

\begin{proof}
The coordinate vector fields of the variation commute, and $t\mapsto\Gamma(t,r,s)$ is a geodesic for every fixed $r$ and $s$. Hence
\begin{equation*}
    \nabla_TT=0,\qquad \nabla_WT=\nabla_TW, \qquad \nabla_VT=\nabla_TV.
\end{equation*}
For fixed $s$, the field $V$ satisfies
\begin{equation}\label{Jacobi equation V}
    \nabla_T\nabla_TV+R(V,T)T=0.
\end{equation}
We differentiate \eqref{Jacobi equation V} covariantly in the $W$ direction. Commuting the covariant derivatives gives
\begin{equation}\label{differentiated Jacobi kinetic term}
\nabla_W\nabla_T\nabla_TV =\nabla_T\nabla_TZ +2R(W,T)\nabla_TV +(\nabla_TR)(W,T)V +R(\nabla_TW,T)V.
\end{equation}
Differentiating the curvature term gives
\begin{equation}\label{differentiated Jacobi curvature term}
    \nabla_W\big(R(V,T)T\big) = (\nabla_WR)(V,T)T+R(Z,T)T + R(V,\nabla_TW)T+R(V,T)\nabla_TW.
\end{equation}
The first Bianchi identity gives
\begin{equation*}
    R(\nabla_TW,T)V+R(V,\nabla_TW)T = R(V,T)\nabla_TW.
\end{equation*}
Substitution of \eqref{differentiated Jacobi kinetic term} and \eqref{differentiated Jacobi curvature term} into the covariant derivative of \eqref{Jacobi equation V}, followed by restriction to $r=s=0$, proves \eqref{inhomogeneous Jacobi equation}.

Since $\Gamma(0,r,s)=e$, we have $V(0,r,s)=0$. Further,
\begin{equation*}
    \frac{DV}{dt}\bigg|_{t=0}=v_0
\end{equation*}
for every $r$ and $s$. Covariantly differentiating these two identities in the $s$ direction gives \eqref{second variation initial conditions}.
\end{proof}

Equation \eqref{inhomogeneous Jacobi equation} is a form of the second-order geodesic deviation equation, cf. Ba{\.z}a{\'n}ski \cite{bazanski1977kinematics}.

\enlargethispage{3\baselineskip}

\begin{remark}
The second variation equation suggests that Proposition~\ref{operator continuity of differential} might alternatively be obtained from suitable mixed $H^{s'}$-$L^2$ estimates for the covariant Hessian of $\riemexp$, together with estimates for the term relating the covariant and ordinary second derivatives.
\end{remark}

\begin{remark}
At a conjugate vector where $d\riemexp$ is Fredholm, the covariant Hessian restricted to the kernel and projected to the cokernel contains information about the singularity type of the exponential map, cf. \cite{warner1965conjugate}. Equation \eqref{inhomogeneous Jacobi equation} may therefore be useful in the study of degenerate conjugate points.
\end{remark}

\bibliographystyle{amsplain}
\bibliography{bibliography.bib}

\enlargethispage{4\baselineskip}

\vfill

\end{document}